\documentclass[12pt,a4paper,twoside]{article}
\usepackage{amssymb,bbm,amsmath,amsthm,amscd}
\usepackage[utf8]{inputenc}
\usepackage{graphicx}
\usepackage{multicol}
\usepackage{multirow}
\usepackage[usenames]{color}
\usepackage{geometry}
\usepackage{hyperref}
\newtheorem{theorem}{Theorem}[section]
\newtheorem{corollary}[theorem]{Corollary}
\newtheorem{lemma}[theorem]{Lemma}
\newtheorem{proposition}[theorem]{Proposition}

\newcommand{\rev}[1]{\textcolor{black}{{#1}}}

\def\d{\displaystyle}

\def\beq{\begin{equation}}
\def\eeq{\end{equation}}
\def\bit{\begin{itemize}}
\def\eit{\end{itemize}}
\def\barr{\begin{array}}
\def\earr{\end{array}}
\def\R{\mathbb{R}}
\def\N{\mathbb{N}}
\def\a{\mathbf{a}}
\def\u{\mathbf{u}}
\def\b{\mathbf{b}}
\def\dt{\Delta t}
\def\unu{u^{n+1}}

\def\uhnu{u_{h}^{n+1}}
\def\utnu{\tilde{u}^{n+1}}
\def\uhn{u_{h}^{n}}
\def\vh{v_h}
\def\vt{\tilde{v}}

\newcommand\redout{\bgroup\markoverwith{\textcolor{red}{\rule[.5ex]{2pt}{0.8pt}}}\ULon}

\def\trh{{\cal T}_{h}}
\def\disp{\displaystyle}
\def\hu{U}

\title{Spectral Variational Multi-Scale method for parabolic problems. Application to 1D transient advection-diffusion equations}
\author{Tom\'as Chac\'on Rebollo\footnotemark[1],  Soledad Fern\'andez-Garc\'ia\footnotemark[1],  \\
David Moreno-Lopez\footnotemark[1],
Isabel S\'anchez Muñoz \footnotemark[2].  }
\begin{document}

\footnotetext[1]
{Dpto. EDAN \& IMUS, University of Seville, Campus de Reina Mercedes, 41012 Sevilla (Spain), e-mail: chacon@us.es,  soledad@us.es,  amorenol@us.es}
\footnotetext[2]
{Dpto. Matem\'atica Aplicada I, University of Seville, Ctra de Utrera s/n, 41013 Sevilla (Spain),
email: isanchez@us.es}
\maketitle

\begin{abstract}
In this work, we introduce a Variational Multi-Scale (VMS) method for the numerical approximation of parabolic problems, where sub-grid scales are approximated from the eigenpairs of associated elliptic operator. The abstract method is particularized 
to the one-dimensional advection-diffusion equations, for which the sub-grid
components are exactly calculated in terms of a spectral expansion when the advection velocity is approximated by piecewise constant velocities on the grid elements.

We prove error estimates that in particular imply that when Lagrange finite element discretisations in space are used, the spectral VMS method coincides with the exact solution of the implicit Euler semi-discretisation of the advection-diffusion problem at the Lagrange interpolation nodes. We also build a feasible method to solve the evolutive advection-diffusion problems by means of an offline/online strategy with reduced computational complexity.

We perform some numerical tests in good agreement with the theoretical expectations, that show an improved accuracy with respect to several stabilised methods.

\end{abstract}

\section{Introduction}
The Variational Multi-Scale is a general methodology to deal with the instabilities arising in the Galerkin discretisation of PDEs (Partial Differential Equations) with terms of different derivation orders (see Hughes (cf. \cite{Hughes_1995,HughesStewart_1995,HughesFMQ_1998})).


The VMS formulation is based upon the formulation of the Galerkin method as two variational problems, one satisfied by the resolved and another satisfied by the sub-grid scales of the solution. To build a feasible VMS method, the sub-grid scales problem is approximately solved by some analytic or computational procedure. In particular, an element-wise diagonalisation of the PDE operator leads to the Adjoint Stabilised Method, as well as to the Orthogonal Sub-Scales (OSS) method, introduced by Codina in \cite{codina12}.  Within these methods, the effects of the sub-grid scales is modelled by means of a dissipative interaction of operator terms acting on the resolved scales. The VMS methods have been successfully applied to many flow problems, and in particular to Large Eddy Simulation (LES) models of turbulent flows (cf. \cite{hug112,john0612,Chaconlibro}).

The application of VMS method to evolution PDEs dates back to the 1990s, when the results from \cite{Hughes_1995} were extended to nonsymmetric linear evolution operators, see \cite{HughesStewart_1995}.  The papers \cite{harari,hauke} deal with the spurious oscillations generated in the Galerkin method for parabolic problems due to very small time steps.  The series of articles \cite{fourier1,fourier2,fourier3} deal with transient Galerkin and SUPG methods, transient subgrid scale (SGS) stabilized methods and transient subgrid scale/gradient subgrid scale (SGS/GSGS), making a Fourier analysis for the one-dimensional advection-diffusion-reaction equation.

A stabilised finite element method for the transient Navier-Stokes equations based on the decomposition of the unknowns into resolvable and subgrid scales is considered in \cite{codina2002,codina2007}. Further, \cite{asensio} compares the Rothe method with the so-called Method of Lines, which consists on first, discretise in space by means of a stabilized finite element method, and then use a finite difference scheme to approximate the solution.

More recently, \cite{ChaconDia} introduced the use of spectral techniques to model the sub-grid scales for 1D steady advection-diffusion equations. The basic observation is that the eigenpairs of the advection-diffusion operator may be calculated analitycally on each grid element. A feasible VMS-spectral discretization is then built by truncation of this spectral expansion to a finite number of modes. An enhanced accuracy with respect to preceding VMS methods is achieved.

In \cite{solfergar}, the spectral VMS method is extended to 2D steady advection-diffusion problems. It is cast for low-order elements as a standard VMS method with specific stabilised coefficients, that are anisotropic in the sense that they depend on two grid P\'eclet numbers. To reduce the computing time, the stabilised coefficients are pre-computed at the nodes of a grid in an off-line step, and then interpolated by a fast procedure in the on-line computation.

The present paper deals with the building of the spectral VMS numerical approximation to evolution advection-diffusion equations.  We construct an abstract spectral VMS discretisation of parabolic equations, that is particularised to 1D advection-diffusion equations. The sub-grid components are exactly calculated in terms of spectral expansions when the driving velocity is approximated by piecewise constant velocities on the grid elements. We prove error estimates that in particular imply that when Lagrange finite element discretisations in space are used, the solution provided by the spectral VMS method coincides with the exact solution of the implicit Euler semi-discretisation at the Lagrange interpolation nodes. We also build a feasible method to solve the evolutive advection-diffusion problem by means of an offline/online strategy that pre-computes the action of the sub-grid scales on the resolved scales. This allows to dramatically reduce the computing times required by the method. We further perform some numerical tests for strongly advection dominated flows. The spectral VMS method is found to satisfy the discrete maximum principle, even for very small time steps. A remarkable increase of accuracy with respect to several stabilised methods is achieved.

The outline of the paper is as follows. In Section \ref{se:af}, we describe the abstract spectral VMS discretisation to linear parabolic problems, which is applied to transient advection-diffusion problems in Section \ref{se:cd}. A feasible method is built in Section \ref{se:SVMS_F}, based upon an offline/online strategy. We present in Section \ref{se:nr} our numerical results, and address some conclusions in Section \ref{se:conclusions}.

\section{Spectral VMS method}
\label{se:af}
In this section, we build the spectral VMS discretisation to abstract linear parabolic equation.

Let $\Omega$ a bounded domain in $\R^d$ and $T>0$ a final time. Let us consider two separable Hilbert spaces on $\Omega$, $X$ and $ H$, so that $X \subset H$  with dense and continuous embedding. We denote $(\cdot,\cdot)$ the scalar product in $X$; $X'$ and
$H'$ are the dual topological spaces of $X$ and $H$, respectively, and
$\langle \cdot, \cdot \rangle$  is the duality pairing between $X'$ and $X$.
We identify $H$ with its topological dual $H'$ so that $X \subset H \equiv H' \subset X'$. Denote by ${\cal L}(X)$ the space of bilinear bounded forms on $X$ and
consider $b \in L^1(0,T; {\cal L}(X))$ uniformly bounded and $X$-elliptic with respect to $t\in(0,T)$.

Given the data $f \in L^2(0,T;X')$ and $u_0 \in H$, we consider the following variational parabolic problem:
\begin{equation}
\label{eq:af}
\left\{\begin{array}{l}
\mbox{Find } u \in L^2((0,T);X) \cap C^0([0,T];H) \mbox{ such that,}
\\[0,2cm]
\d\frac{d}{dt}(u(t),v) \,+\, b(t;u(t),v) \,=\, \langle f(t),v\rangle \quad \forall \, v \in X,
\quad \mbox{in  } {\cal D}'(0,T),
\\[0,4cm]
u(0) \,=\, u_0 \quad \mbox{in  } H.
\end{array}\right.
\end{equation}
It is well known that this problem is well posed and, in particular, admits a unique solution \cite{dautray}.
To discretize this problem, we proceed through the so-called Horizontal Method of Lines \cite{asensio,bernardi,hauke}. First, we discretise in time by the Backward Euler scheme and then we apply a steady spectral VMS method to the elliptic equations appearing at each time step.

Consider a uniform partition of the interval $[0,T]$, $\{0=t_0<t_1<...<t_N=T\}$,  with time-step size $\Delta t=T/N$. The time discretization of problem (\ref{eq:af}) by the Backward Euler scheme gives the following family of stationary problems: given the initialization  $u^0 \,=\, u_0$,
\begin{equation}
\label{eq:discreto_t}
\left\{\begin{array}{l}
\mbox{Find } \unu\in X \mbox{ such that,}
\\[0,2cm]
\left( \displaystyle \frac{u^{n+1}-u^n}{\Delta t}, v\right) + b^{n+1}(u^{n+1},v) \,=\, \dt\,  \langle f^{n+1},v \rangle   \quad \forall \, v\in X,
\\[0,4cm]
\forall \,  n=0,1, \hdots,N-1,
\end{array}\right.
\end{equation}

where $b^{n+1}$ and $f^{n+1}$ are some approximations of $b(t;\cdot,\cdot)$ and $f(t)$, respectively,  at $t=t_{n+1}$.

\vspace{4mm}

To discretise in space problem (\ref{eq:discreto_t}), we assume that $\Omega$ is polygonal (when $d=2$) or polyhedric (when $d=3$), and consider
a family of conforming and regular triangulations of $\overline{\Omega}$, $\{{\cal T}_{h}\}_{h>0}$, formed by simplycial elements, where the parameter $h$ denotes the largest diameter of the elements of the triangulation ${\cal T}_{h}$.  The VMS method is based on the decomposition,
\begin{equation*}
X=X_h\oplus \tilde{X},
\end{equation*}
where $X_h$ is a continuous finite element sub-space of $X$ constructed on the grid ${\cal T}_{h}$, and $\tilde{X}$ is a complementary, infinite-dimensional, sub-space of $X.$ Notice that this is a multi-scale decomposition of the space $X$, being $X_h$ the large or resolved scale space and
$\tilde{X}$ the small or sub-grid scale space.  This decomposition defines two projection operators $P_h: X \mapsto X_h$ and $\tilde{P}: X \mapsto \tilde{X}$, by
\begin{equation}
\label{proyectores}
P_h (v)= v_h ,\quad \tilde{P}(v)=\tilde{v}, \quad \forall \, v\in X,
\end{equation}
where $v_h$ and $\tilde{v}$ are the unique elements belonging to $X_h$ and $\tilde{X}$, respectively, such that \mbox{$v= v_h + \tilde{v}$}. Hence, one can decompose the solution of problem (\ref{eq:discreto_t}) as
\begin{equation*}
\label{ude}
u^{n+1}=\uhnu +\utnu,
\end{equation*}
where $\uhnu=P_h (u^{n+1})$ and $\utnu=\tilde{P}(u^{n+1})$ satisfy the coupled problem,
\begin{equation}
\label{VMS} \nonumber
\left\{  \begin{array}{ll}
\displaystyle  \left( \frac{u^{n+1}_h-u^n_h}{\Delta t},v_h \right)  + \left( \frac{\utnu-\tilde{u}^n}{\Delta t}, v_h \right) + b^{n+1}(u^{n+1}_h,v_h)  + b^{n+1}(\utnu ,v_h)
 = \langle f^{n+1},v_h \rangle & (\ref{VMS}.1)
 \\ [0,6cm]
\displaystyle  \left( \frac{u^{n+1}_h-u^n_h}{\Delta t},\tilde{v} \right)  + \left( \frac{\utnu-\tilde{u}^n}{\Delta t}, \tilde{v} \right) + b^{n+1}(u^{n+1}_h,\tilde{v})  + b^{n+1}(\utnu ,\tilde{v})
 = \langle f^{n+1},\tilde{v} \rangle & (\ref{VMS}.2)
\\ [0,5cm]
\forall v_h \in X_h, \, \forall \tilde{v}\in \tilde{X},
\end{array}\right.
\end{equation}
for all $ \,  n=0,1, \hdots,N-1$. The small scales component $\utnu$ thus satisfies,
\begin{equation}
\label{eq:smalls}
(\utnu, \tilde{v})   + \Delta t \  b^{n+1}(\utnu ,\tilde{v})
= \langle R^{n+1}(u_h^{n+1}),\tilde{v}\rangle
\end{equation}
where $\langle R^{n+1}(u_h^{n+1}),\tilde{v}\rangle $ is the residual of the large scales component, defined as,
\begin{equation*}
\begin{array}{r}
\langle R^{n+1}(\uhnu), \vt \rangle := (u^{n}_h+ \tilde{u}^{n},\tilde{v}) +  \Delta t \ \langle f^{n+1},\tilde{v}\rangle - (u^{n+1}_h, \tilde{v}) -  \Delta t  \ b^{n+1}(u^{n+1}_h,\tilde{v}),
\,\forall \, \vt\in \tilde{X}.
\end{array}
\end{equation*}
In condensed notation, this may be written as,
\begin{equation}
\label{eq:smallsPi}
\utnu=\Pi^{n+1}(R^{n+1}(\uhnu)),
\end{equation}
where
\begin{equation*}
\begin{array}{cccl}
 \Pi^{n+1}:  & \tilde{X} & \rightarrow & \tilde{X} \\
 			& g 			& \mapsto			& \Pi^{n+1}(g) = \tilde{G}
\end{array}
\end{equation*}
is the static condensation operator on $\tilde{X}$ defined as,
\begin{equation*}
\label{pi}
( \tilde{G} ,  \tilde{v}) +\Delta t \  b^{n+1}(\tilde{G},\tilde{v})=\langle g, \vt \rangle \quad \forall\,  \vt \in \tilde{X},
\mbox{ for any } g \in \tilde{X}'.
\end{equation*}
Inserting expression (\ref{eq:smallsPi}) in the large scales equation (\ref{VMS}.1), leads to the condensed VMS formulation of problem (\ref{eq:discreto_t}):
\begin{equation}
\label{eq:VMSd}
\left\{\begin{array}{l}
\mbox{Find } \uhnu\in X_h \mbox{ such that}
\\[0,3cm]
(u^{n+1}_h, v_h) + \Delta t \, b^{n+1}(u^{n+1}_h,v_h)  +  (\Pi^{n+1}(R^{n+1}(\uhnu)), \vh)+\Delta t \,b^{n+1}(\Pi^{n+1}(R^{n+1}(\uhnu)), \vh)
\\[0,4cm]
 \quad\quad =\dt\,  \langle f^{n+1},v_h \rangle   \,+\, (u^{n}_h+ \Pi^n ( R^n ( u^n_h) ),v_h)
\\[0,3cm]
\forall \, \vh\in X_h,\,\,\forall \,  n=0,1, \hdots,N-1,
\end{array}\right.
\end{equation}
with $u_h^0=P_h(u_0)$. This problem is an augmented Galerkin formulation, where the additional terms represents the effect of the small scales component of the solution $(\utnu)$ on the large scales component $(\uhnu)$.

To build an approximation of the sub-grid scales, we use a spectral decomposition of the ope\-rator associated to the variational formulation on each grid element, at each discrete time. To apply this approxi\-mation to problem (\ref{eq:VMSd}), the small scales space
$\tilde{X}$ is approximated by the \lq\lq bubble" sub-spaces,
\begin{equation}
\label{eq:Xth}
\tilde{X}\simeq \tilde{X}_h = \bigoplus_{K\in\mathcal{T}_h}\tilde{X}_K,
\quad\mbox{with } \tilde{X}_K=\{\tilde{v}\in\tilde{X}, \mbox{ such that } \mbox{supp}(\tilde{v})\subset K\}.
\end{equation}
Hence, we approximate
\begin{equation}
\label{eq:utK}
\utnu \simeq \utnu_h = \sum_{K\in\mathcal{T}_h}\utnu_K, \quad \mbox{with }\utnu_K \in \tilde{X}_K,
\quad \forall \,  n=0,1, \hdots,N-1.
\end{equation}
Then, problem (\ref{eq:smalls}) is approximated by the following family of decoupled problems,
\begin{equation}
\label{eq:smallsK}
\begin{array}{l}
(\tilde{u}_K^{n+1},\tilde{v}_K) + \Delta t \ b^{n+1}(\tilde{u}_K^{n+1},\tilde{v}_K)=\langle R^{n+1}(u_h^{n+1}),\tilde{v}_K\rangle,
\quad\forall\, \tilde{v}_K\in\tilde{X}_K, \quad \forall \, K\in\mathcal{T}_h.
\end{array}
\end{equation}
Let $\mathcal{L}^{n+1}:  X \mapsto X'$ be the operator defined by
\begin{equation}
\label{eq:opn}
\langle \mathcal{L}^{n+1} w,v \rangle=b^{n+1}(w,v), \quad \forall \, w, v \in X,
\end{equation}
and let $\mathcal{L}^{n+1}_K$ be the restriction of this operator to $\tilde{X}_K$. Let us also consider
the weighted  $L^2$ space,
\begin{equation*}
L^2_p(K)=\{w:K\rightarrow \mathbb{R} \mbox{ measurable such that } p|w|^2\in L^1(K)\},
\end{equation*}
where $p$ is some measurable real function defined on $K,$ which is positive a.e. on $K$. This is a Hilbert space endowed with the inner product
\begin{equation*}
(w,v)_p=\int_K p(x) w(x) v(x) dx.
\end{equation*}
We denote by $\|\cdot\|_{p}$ the norm on $L^2_p(K)$ induced by this inner product.

\vspace{2mm}

Now,  we can state the following result,  which allows to compute the small scales on each grid element by means a spectral expansion.
\begin{theorem}
\label{th:main}
Let us assume that there exists a complete sub-set $\{\tilde{z}_j^{n,K}\}_{j\in\mathbb{N}}$ on
$\tilde{X}_K$ formed by eigenfunctions of the operator $\mathcal{L}^{n}_K$, which is an orthonormal system in $L^2_{p^{n,K}}(K)$ for some weight function $p^{n,K}\in C^1(\bar{K}).$ Then,
\begin{equation}
\label{eq:serie}
\begin{array}{l}
\tilde{u}_K^{n}=\d\sum_{j=1}^{\infty}\beta_j^{n,K} \, r^{n,K}_j \, \tilde{z}_j^{n,K}, \quad \forall \, n=1,\ldots,N,
\end{array}
\end{equation}
where
$\beta_j^{n,K} = (\Lambda_j^{n,K})^{-1}$,  with $\Lambda_j^{n,K}=1+\Delta t \, \lambda_j^{n,K}$ being $\lambda_j^{n,K}$ the eigenvalue of
$\mathcal{L}^{n}_K$ associated to $\tilde{z}_j^{n,K}$, and
$$
r^{n,K}_j = \langle R^{n}(u_h^{n}), p^{n,K}\,\tilde{z}_j^{n,K}\rangle.
$$
\end{theorem}
This is a rather straightforward application of Theorem 1 in \cite{ChaconDia}, that we do not detail for brevity.

Once the eigenpairs $(\tilde{z}_j^{n+1,K},\lambda_j^{n+1,K})$ are known, the previous procedure allows us to directly compute $u_h^{n+1}$ from problem (\ref{eq:VMSd}), approximating the sub-grid component $\tilde{u}^{n+1}$ by expressions (\ref{eq:utK}) and (\ref{eq:serie}). This gives the spectral VMS method to fully discretize problem (\ref{eq:af}). Namely,
\begin{equation}
\label{eq:SVMS}
\left\{ \begin{array}{l}
\mbox{Find } \uhnu\in X_h \mbox{ such that}
\\[0,3cm]
(u^{n+1}_h, v_h) + \Delta t \,  b^{n+1}(u^{n+1}_h,v_h)  + (\utnu_h, v_h) + \Delta t \, b^{n+1}(\utnu_h ,v_h)
\\[0,4cm]
 \qquad = \dt\,  \langle f^{n+1},v_h \rangle  + (u^{n}_h,v_h) + (\tilde{u}_h^{n}, v_h)
\\ [0,3cm]
\forall \, v_h \in X_h, \quad \forall \,  n=0,1,\hdots,N-1,
\end{array}\right.
\end{equation}
where,
\begin{equation}
\label{eq:smalls_d}
\tilde{u}^{n+1}_h = \sum_{K\in\mathcal{T}_h} \sum_{j=1}^{\infty}
\beta_j^{n+1,K} \, \langle R^{n+1}_h(u_h^{n+1}), p^{n+1,K} \,\tilde{z}_j^{n+1,K}\rangle \, \tilde{z}_j^{n+1,K},
\quad \forall \,  n=0, \hdots,N-1,
\end{equation}
with
\begin{equation*}
\langle R^{n+1}_h(\uhn), \vt \rangle :=
(u^{n}_h+\tilde{u}_h^{n},\tilde{v}) + \Delta t \, \langle f^{n+1},\tilde{v}\rangle
- (u^{n+1}_h, \tilde{v}) -  \Delta t  \ b^{n+1}(u^{n+1}_h,\tilde{v}),
\quad\forall \, \vt\in \tilde{X},
\end{equation*}
$u_h^0 \,=\, P_h(u_0)$ and $\tilde{u}_h^0 \in\tilde{X}_h$ some approximation of $\tilde{u}^0$.

\section{Application to transient advection-diffusion problems}
\label{se:cd}
In this section, we apply the abstract spectral VMS method introduced in the previous section to transient advection-diffusion equations, that we state with homogeneous boundary conditions,
\begin{equation}
\label{eq:cd}
\left\{\begin{array}{ll}
\partial_t u +\mathbf{a}\cdot\nabla u-\mu \Delta u = f &
\mbox{in }\Omega\times(0,T),
\\[0,2cm]
u=0 & \mbox{on }\partial \Omega\times(0,T),
\\[0,2cm]
u(0)=u_0 & \mbox{on } \Omega,
\end{array}\right.
\end{equation}
where
$\mathbf{a} \in L ^\infty(0,T;W^{1,\infty}(\Omega))^d$ is the advection velocity field, $\mu>0$ is the diffusion coefficient, $f\in L^2((0,T);L^2(\Omega))$ is the source term and $u_0\in L^2(\Omega)$ is the initial data. Different boundary conditions may be treated as well, as these also fit into the general spectral VMS method introduced in the previous section.

The weak formulation of problem \eqref {eq:cd} reads,
\begin{equation}
\label{eq:weak_cd}
\left\{\begin{array}{l}
\mbox{Find }u\in L^2((0,T);H^1_0(\Omega))\cap C^0([0,T];L^2(\Omega)) \mbox{ such that,}
\\[0,2cm]
\d\frac{d}{dt}(u(t),v) + (\mathbf{a}\cdot\nabla u(t),v) + \mu (\nabla u(t),\nabla v)=
\langle f(t),v\rangle \quad \forall \, v\in H_0^1(\Omega),
\\[0,3cm]
u(0)=u_0.
\end{array}\right.
\end{equation}
Problem (\ref{eq:weak_cd})  admits the abstract formulation (\ref{eq:af}) with
$H=L^2(\Omega)$, $X=H^1_0(\Omega)$ and
\begin{equation*}
b(w,v)=(\mathbf{a}\cdot\nabla w,v)+\mu (\nabla w,\nabla v),
\quad \forall\, w,v\in H^1_0(\Omega).
\end{equation*}

In practice, we replace the velocity field $\a$ by $\a_h$, the piecewise constant function defined a. e. on $\overline{\Omega}$ such that $\a_h =\a_K$ on the interior of each element $K\in\trh$. Then, we apply the spectral VMS method to the approximated problem,
\begin{equation}
\label{vmstilde}
\left\{\begin{array}{ll}
\mbox{Find } \hu^{n+1}\in H^1_0(\Omega)\mbox{ such that}
\\[0,3cm]
\noindent \begin{array}{r}
\left( \displaystyle \frac{\hu^{n+1}-\hu^n}{\Delta t}, v\right) + (\a_h^{n+1}\cdot\nabla \hu^{n+1},v) + \mu (\nabla \hu^{n+1},\nabla v)=
\langle f^{n+1},v\rangle,  \;\forall \, v\in H^1_0(\Omega),
\end{array}
\\[0,5cm]
\forall \,  n=0,1, \hdots,N-1,
\end{array}\right.
\end{equation}
with $u^0=u_0$.

In this case, $\mathcal{L}^n w =\mathbf{a}_h^n \cdot\nabla w - \mu \Delta w$
is the advection-difusion operator. Proposition 1 in \cite{ChaconDia} proved that the eigenpairs $(\tilde{w}_j^{n,K}, \lambda_j^{n,K})$ of operator $\mathcal{L}^n_K$ can be obtained from the eigenpairs $(\tilde{W}_j^K, \sigma_j^K)$ of the Laplace operator in $H_0^1(K)$, in the following way:
\begin{equation}
\label{eq:eigen}
\barr{l}
\tilde{w}_j^{n,K} = \psi^{n,K} \, \tilde{W}_j^K, \quad \psi^{n,K}(x)= \exp\left(\frac{1}{2\mu} \,\a_K^n\cdot x \right)
\\[0,2cm]
\lambda_j^{n,K} = \mu \, \left( \sigma_j^K + \d\frac{|\a_K^n|^2}{4\mu^2} \right),
\quad \forall \, j\in\N.
\earr
\end{equation}
Moreover, for the weight function
\begin{equation}
\label{eq:pk}
\d p^{n,K}(x)= (\psi^{n,K})^{-2}= \exp\left(-\frac{1}{\mu}\,\a_K\cdot x\right)
\end{equation}
the sequence
\begin{equation}
\label{eq:zjk}
\tilde{z}_j^{n,K} = \d\frac{\tilde{w}_j^{n,K}}{\|\tilde{w}_j^{n,K}\|_{p^{n,K}}}, \quad \forall \, j\in\N,
\end{equation}
is a complete and orthonormal system in $L_{p^{n,K}}^2(K)$ (see  Theorem 2 in \cite{ChaconDia}).
Then,  Theorem \ref{th:main} holds and it is possible to apply the method (\ref{eq:SVMS}) to problem (\ref{vmstilde}).

\subsection{One dimension problems}
The eigenpairs of the Laplace operator can be exactly computed for grid elements with simple geometrical forms, as it is the case of parallelepipeds. In the 1D case, the elements $K\in \trh$ are closed intervals, $K=[a,b]$. The eigenpairs
$(\tilde{W}_j^K, \sigma_j^K)$ are solutions of the problem
$$\left\{
\barr{l}
-\partial_{xx} \tilde{W}^K = \sigma^K \, \tilde{W}^K \,\, \mbox{in } K,
\\[0,2cm]
\tilde{W}^K(a)=\tilde{W}^K(b)=0.
\earr\right.
$$
Solutions of this problem are
$$
\tilde{W}_j^K = \sin \big( \sqrt{\sigma_j^K} \, (x- a) \big), \quad
\sigma_j^K= \left(\frac{j\pi}{h_K} \right)^2, \,\, \mbox{with  } h_K=b-a,\,\, \mbox{for any } j\in\N.
$$
As the function $p^{n,K}$ defined in \eqref{eq:pk} is unique up to a constant factor, to express the eigenpairs in terms of non-dimensional parameters, we replace $p^{n,K}$  by (we still denote it in the same way),
\begin{equation}
\label{pk1d}
p^{n,K}(x)= \exp\left(-2\, P_{n,K}\, \frac{x-a}{h_K}\right ),
\end{equation}
where $P_{n,K}=\d\frac{|\a_K^n|\,h_K}{2\mu}$ is the element P\'eclet number. Then, from expressions (\ref{eq:eigen}) and (\ref{eq:zjk}),
\begin{equation}
\label{eig1d}
\tilde{z}_j^{n,K}=\sqrt{\frac{2}{h_K}} \exp\left(P_{n,K} \frac{x- a}{h_K}\right) \,
\sin \left(j\pi \frac{x- a}{h_K}\right), \quad
\lambda_j^K= \mu \, \left(\frac{j\pi}{h_K} \right)^2 + \d\frac{|\a_K^n|^2}{4\mu}.
\end{equation}
 It follows
\begin{equation}
\label{beta1d}
\d  \beta_j^{n,K} = \frac{1}{1+S_K(P_{n,K}^2+\pi^2j^2)}    \quad   \mbox{for any } j\in\N,
\end{equation}
where $S_K=\d\frac{\Delta t\, \mu}{h_K^2}$ is a non-dimensional parameter that represents the relative strength of the time derivative and diffusion terms in the discrete equations, at element $K$.

\subsection*{Error analysis}
We afford in this section the error analysis for the solution of the 1D evolutive convection-diffusion problem by the spectral VMS method \eqref{eq:SVMS}.

Let $\{\alpha_i\}_{i=0}^I \in \bar{\Omega}$ be the Lagrange interpolation nodes of space $X_h$. Let $\omega_i=(\alpha_{i-1},\alpha_i)$, $i=1,\ldots, I$. Setting $\tilde{X}_i=H^1_0(\omega_i)$, it holds,
$$
H^1_0(\Omega)= X_h \oplus \tilde{X},\quad \mbox{with  }\, \tilde{X}=\bigoplus_{i=1}^I \tilde{X}_i.
$$

Observe that this decomposition generalises (\ref{eq:Xth}) with $\tilde{X}_h=\tilde{X}$. Moreover, when operator
in (\ref{eq:opn}) is $\mathcal{L}^n w =\a_h^n\cdot\nabla w - \mu \Delta w$,
problem (\ref{eq:smalls}) can be exactly decoupled into the family of problems (\ref{eq:smallsK}).
In particular, if the projection operator $P_h$ in \eqref{proyectores} is the Lagrange interpolate on $X_h$, then
$\hu_h^n =P_h(\hu^n)$, $\tilde{\hu}_h^n= \hu^n- \hu_h^n \in  \tilde{X}$ and consequently, $\hu_h^n \in X_h$ satisfies method \eqref{eq:SVMS}.

Notice that thanks to the spectral expansion, the sub-grid scales contribution in method \eqref{eq:SVMS}, when the advection velocity is element-wise constant, is exactly computed, and then, the discretisation error only is due to the time discretisation and the approximation of the advection velocity $\a$, but not to the space discretisation.

Therefore, to analize the discretisation error we compare the solution of problem \eqref{vmstilde} to the solution of the implicit Euler time semi-discretisation of problem \eqref{eq:weak_cd},
\begin{equation}
\label{eq:discreto_t3}
\left \{ \begin{array}{l}
\mbox{Find } u^{n+1}\in H^1_0(\Omega)\mbox{ such that}
\\[0,3cm]
\left( \displaystyle \frac{u^{n+1}-u^n}{\Delta t}, v\right) + (\a^{n+1}\partial_x u^{n+1},v) + \mu \,(\partial_x u^{n+1},\partial_x v)=
\langle f^{n+1},v\rangle  \quad \forall \, v\in H^1_0(\Omega),
\\[0,3cm]
\forall \,  n=0,1, \hdots,N-1,
\end{array}\right.
\end{equation}
with $u^0=u_0$.

We assume that $\a_h$ restricted to each $K$ is extended by continuity to $\partial K$. Given a sequence $b=\{b^n,\,n=1,\cdots, N\} $ of elements of a normed space $Y$, let us denote,
$$
\|b\|_{l^p(Y)}=\left (\Delta t \,\sum_{n=1}^N \|b^n\|_Y^p \right )^{1/p},\quad
\|b\|_{l^\infty(Y)}=\max_{n=1,\cdots,N} \|b^n\|_Y.
$$

We shall use the following discrete Gronwall's lemma, whose proof is standard, and so we omit it.
\begin{lemma}\label{lemagrom}
Let $\alpha_n$, $\beta_n$, $\gamma_n$, $n=1,2,...$ be non-negative real numbers such that
\begin{equation*} \label{estaux1}
(1-\sigma \,\Delta t) \,\alpha_{n+1} + \beta_{n+1} \le (1+\tau \,\Delta t) \,\alpha_n + \gamma_{n+1}
\end{equation*}
for some $\sigma \ge 0$, $\tau \ge 0$. Assume that $\sigma \, \Delta t \le 1-\delta$ for some $\delta >0$. Then it holds
\begin{equation*} \label{estgr1}
\alpha_n \le e^{\rho \,t_n}\,\alpha_0 + \frac{1}{\delta} \,\sum_{l=1}^n e^{\rho\, (t_n-t_l)}\, \gamma_l,
\end{equation*}
and
\begin{equation*} \label{estgr2}
\sum_{l=1}^n\beta_l \le \left (1+\frac{\tau}{\sigma} +(\sigma+\tau)\,e^{\rho\, t_{n-1}}\,t_{n-1} \right )\,\alpha_0 +\frac{1}{\delta} \, \left ( 1+ (\sigma+\tau)\,e^{\rho\, t_{n-1}}\,t_{n-1} \right )\sum_{l=1}^n\,\gamma_l,
\end{equation*}
with $\displaystyle \rho=(\sigma+\tau)/\delta$.
\end{lemma}
 Let $e=\{e^n,\,n=0,1,\cdots, N\} \subset H^1_0(\Omega)$ be the sequence of errors $e^n=u^n - \hu^n \in H^1_0(\Omega)$, where we recall that $\hu^n$  is the solution of the discrete problem \eqref{vmstilde}, and denote $\delta_t e^{n+1}=\disp\frac{e^{n+1}-e^n}{\Delta t}$. It holds the following result.
\begin{proposition} \label{teoerrest}
Assume that $\a\in L^\infty(\Omega \times (0,T))^d$, $f\in L^2(\Omega \times (0,T))$, $\disp \Delta t \le (1-\varepsilon)\, \frac{\mu}{\|\a\|_{L^\infty(\Omega \times (0,T))}^2}$  for some $\varepsilon \in (0,1)$ and $\|\a_h\|_{L^\infty(\Omega \times (0,T))} \le D\, \|\a\|_{L^\infty(\Omega \times (0,T))}$ for some constant $D>0$.
Then,
\begin{equation} \label{esterr2}
\|\delta_t e\|_{l^2(L^2(\Omega))}+ \mu \|e\|_{l^\infty(H^1_0(\Omega))} \le C\, \|\a_h -\a\|_{l^2(L^\infty(\Omega))},
\end{equation}
for some constant $C>0$ independent of $h$, $\Delta t$ and $\mu$.
\end{proposition}
\begin{proof} 
Let us substract \eqref{vmstilde}  from \eqref{eq:discreto_t3} with $v=v_h\in X_h$. This yields
\begin{equation*}
\label{eq:ecuerror}
\begin{array}{l}
\left( \displaystyle \frac{e^{n+1}-e^n}{\Delta t}, v_h\right) + (\a_h^{n+1}\partial_x e^{n+1},v_h) + \mu \,(\partial_x e^{n+1},\partial_x v_h)=
((\a_h^{n+1}-\a^{n+1})\partial_x u^{n+1},v_h) .
\end{array}
\end{equation*}
Setting $v_h = \delta_t e^{n+1}$, and using the identity $2(b,b-a)=\|b\|_{L^2(\Omega)}^2-\|a\|_{L^2(\Omega)}^2+\|b-a\|_{L^2(\Omega)}^2$ for any $a,\, b \in {L^2(\Omega)}^d$ yields
\begin{eqnarray}
\Delta t \,\|\delta_t e^{n+1}\|^2_{L^2(\Omega)} +\Delta t \, (\a_h^{n+1}\partial_x e^{n+1},\delta_t e^{n+1})&+&\frac{\mu}{2}\, \left (\|\partial_x  e^{n+1}\|^2_{L^2(\Omega)}- \|\partial_x  e^n\|^2_{L^2(\Omega)} \right )\nonumber \\\label{estuno}
&\le&\Delta t \,((\a_h^{n+1}-\a^{n+1})\partial_x u^{n+1},\delta_t e^{n+1}).
\end{eqnarray}
It holds
\begin{eqnarray}
|(\a_h^{n+1}\,\partial_x e^{n+1},\delta_t e^{n+1})| &\le&   \|\a_h^{n+1}\|_{L^\infty(\Omega)}\,\|\partial_x e^{n+1}\|_{L^2(\Omega)}\, \|\delta_t e^{n+1}\|_{L^2(\Omega)}\nonumber \\
&\le&\frac{1}{2} \|\delta_t e^{n+1}\|^2_{L^2(\Omega)}+\frac{ \|\a\|_{L^\infty(\Omega \times (0,T))}^2}{2}\,\|\partial_x e^{n+1}\|^2_{L^2(\Omega)}.\label{estdos}
\end{eqnarray}
As $\a\in L^\infty(\Omega \times (0,T))^d$, $f\in L^2(\Omega \times (0,T))$, then the $u^n$ are uniformly bounded in $L^\infty(0,T;H^1_0(\Omega))$, due to the standard estimates for the implicit Euler method in strong norms. Then, for some constant $C>0$,
\begin{eqnarray}
((\a_h^{n+1}-\a^{n+1})\partial_x u^{n+1},\delta_t e^{n+1})&\le&   \|\a_h^{n+1}-\a^{n+1}\|_{L^\infty(\Omega)}\,\|\partial_x u^{n+1}\|_{L^2(\Omega)}\, \|\delta_t e^{n+1}\|_{L^2(\Omega)}\nonumber \\
&\le&C \, \|\a_h^{n+1}-\a^{n+1}\|_{L^\infty(\Omega)}^2 + \frac{1}{4}\, \|\delta_t e^{n+1}\|^2_{L^2(\Omega)}.\label{estres}
\end{eqnarray}
Hence, combining \eqref{estdos} and \eqref{estres} with \eqref{estuno},
$$
\frac{\Delta t}{4} \,\|\delta_t e^{n+1}\|^2_{L^2(\Omega)} +\frac{\mu}{2}\, ( 1- \sigma \,\Delta t )\, \|\partial_x  e^{n+1}\|^2_{L^2(\Omega)}\le \frac{\mu}{2}\, \|\partial_x  e^n\|^2_{L^2(\Omega)} + C \, \Delta t\,\|\a_h^{n+1}-\a^{n+1}\|_{L^\infty(\Omega)}^2,
$$
with $\sigma =\disp\frac{\|\a\|_{l^\infty(L^\infty(\Omega))}^2}{\mu}$. Applying the discrete Gronwall's lemma \ref{lemagrom},  estimate \eqref{esterr2} follows.
\end{proof}
\begin{corollary} \label{corol}Under the hypotheses of Proposition \ref{teoerrest}, it holds
\begin{equation} \label{estinfinf}
\mu\,\|  e^n\|_{l^\infty(L^\infty(\Omega))}\le C\, \|\a_h -\a\|_{l^2(L^\infty(\Omega))}
\end{equation}
for some constant $C>0$. Moreofer, if $\a$ is constant, then the solution $\hu_h^n$  of the spectral VMS method \eqref{eq:SVMS} coincides with the solution $u^n$ of the implicit Euler time semi-discretisation  \eqref{eq:discreto_t3} at the Lagrange interpolation nodes of space $X_h$.
\end{corollary}
\begin{proof}
In one space dimension $H^1(\Omega) $ is continuously injected in $L^\infty(\Omega)$. Then estimate \eqref{estinfinf} follows from estimate \eqref{esterr2}.

If $\a$ is constant obviously $\hu^n = u^n$ for all $n=0,1,\cdots, N$. As $\hu^n_h(\alpha_i) = \hu^n(\alpha_i)$ at the Lagrange interpolation nodes $\alpha_i$, $i=1,\ldots,I$, then $\hu_h^n$ coincides with $u^n$ at these nodes.
\end{proof}

\section{Feasible method: offline/online strategy}\label{se:SVMS_F}
Building the spectral VMS method using the formulation \eqref{eq:SVMS}  requires quite large computing times, due to the summation of the spectral expansions that yield the coefficients of the matrices that appear in the algebraic expression of the method.

In order to reduce this time, we shall neglect the dependency of method \eqref{eq:SVMS} w.r.t.  $\tilde{u}^{n-1}$.

Then, our current discretization of problem \eqref{eq:af} is the following,
\begin{equation}
\label{eq:SVMS_F}
\left\{ \begin{array}{l}
\mbox{Find } \uhnu\in X_h \mbox{ such that}
\\[0,3cm]
(u^{n+1}_h, v_h) + \Delta t \,  b^{n+1}(u^{n+1}_h,v_h)  + (\utnu_h, v_h) + \Delta t \, b^{n+1}(\utnu_h ,v_h)
\\[0,4cm]
 \qquad = \dt\,  \langle f^{n+1},v_h \rangle  + (u^{n}_h,v_h) + (\tilde{u}_h^{n}, v_h)
\\ [0,3cm]
\forall \, v_h \in X_h, \quad \forall \,  n=0,1,\hdots,N-1,
\end{array}\right.
\end{equation}
where $\tilde{u}^{n+1}_h$ is given by \eqref{eq:smalls_d}, but $\tilde{u}^{n}_h$ is defined from an approximated residual:
\begin{equation}
\label{eq:uhtn_F}
\tilde{u}^{n}_h = \sum_{K\in\mathcal{T}_h} \sum_{j=1}^{\infty}
\beta_j^{n,K} \, \langle \hat{R}^n_h(u_h^{n}), p^{n,K} \,\tilde{z}_j^{n,K}\rangle \, \tilde{z}_j^{n,K}
\end{equation}
with
\begin{equation*}
\label{eq:res_g}
\langle \hat{R}^n_h({u}^n_h), \tilde{v} \rangle = ({u}^{n-1}_h,\tilde{v})+ \Delta t \, \langle f^n,\tilde{v}\rangle - ({u}^n_h,\tilde{v})-\Delta t\, b^n({u}^n_h,\tilde{v}),\quad \forall \tilde{v} \in \tilde{X}.
\end{equation*}

Neglecting the dependency of method \eqref{eq:SVMS} w.r.t.  $\tilde{u}^{n-1}$ allows to eliminate the recurrence in time of the sub-grid scales. Thanks to this fact, problem \eqref{eq:SVMS_F} is equivalent to a linear system (that we  describe in detail in Appendix), whose coefficients only depend on non-dimensional parameters.

\subsection{Application to 1D transient advection-diffusion problems}
In this case the coefficients of the linear system equivalent to problem \eqref{eq:SVMS_F} only depend on two non-dimensional parameters, as we confirm below.

As we can see in Appendix, if $\left\{ \varphi_m\right\}_{m=1}^{L+1}$ is a basis of the space $X_h$ associated to a partition  $\{ x_1 < x_2 < \ldots < x_{L+1}\}$ of $\Omega$, the solution $u_{h}^{n+1}$ of \eqref{eq:SVMS_F} can be written as
\begin{equation*}
u_h^{n+1} = \d\sum_{m=1}^{L+1} u_m^{n+1}\varphi_m.
\end{equation*}
Then, the unknown vector $\mathbf{u}^{n+1}=(u_1^{n+1},u_2^{n+1},\hdots,u_L^{n+1},u_{L+1}^{n+1})^t\in\R^{L+1} $ is the solution of the linear system
\beq
\label{eq:sl}
\mathbf{A}^{n+1} \, \u^{n+1} = \b^{n+1},
\eeq
where the matrix and second term are defined in \eqref{eq:csl} from matrices $A_i^{n+1}$ and $B_i^{n+1}$ given by 
\eqref{a1}-\eqref{a4} and \eqref{b1}-\eqref{b4}.

We focus, for instance, on the coefficients of matrix $A^n_1$:
$$
(A^n_1)_{lm}= \d\sum_{K\in\mathcal{T}_h} \sum_{j=1}^{\infty} \beta_j^{n,K}
(\varphi_m, p^{n,K} \tilde{z}_j^{n,K})(\tilde{z}_j^{n,K}, \varphi_l ).
$$
Let $K=[x_{l-1},x_l]\in \trh$. From expressions \eqref{pk1d} and \eqref{eig1d}, $p^{n,K}$ and $\tilde{z}_j^{n,K}$ depend on the element non-dimensional parameters $P_{n,K}$ and $S_K$ and the non-dimensional variable $\disp \hat{x}=\frac{x-x_{l-1}}{h_K}$. The change of variable $\hat{x} \in [0,1] \mapsto x\in K$ from the reference element $[0,1]$ to element $K$ in the integral expressions
$$(\varphi_m, p^{n,K} \tilde{z}_j^{n,K})=\int_K \varphi_m\, p^{n,K}(x) \tilde{z}_j^{n,K}(x) \, dx,\quad (\tilde{z}_j^{n,K}, \varphi_l ) = \int_K \tilde{z}_j^{n,K}(x)\, \varphi_l(x)\, dx
$$
readily proves that these expressions (up to a factor depending on $h$) can be written as functions of $S_K$ and $P_K$.
Further, by \eqref{beta1d} the coefficients $\beta_j^{n,K}$ also depend on $P_{n,K}$ and $S_K$. Then, for each $K\in \trh$ the spectral expansion that determines the element contribution to coefficient $(A^n_1)_{lm}$, that is,
 $$
  \sum_{j=1}^{\infty} \beta_j^{n,K}
(\varphi_m, p^{n,K} \tilde{z}_j^{n,K})(\tilde{z}_j^{n,K}, \varphi_l ),
$$
is a function of $P_{n,K}$ and $S_K$, up to a factor depending on $h$. This also holds for the coefficients of all other matrices that defines the linear system \eqref{eq:sl}, $A^n_i$ and $B^n_i$, as these are built from the basic values
$(\varphi_m, p^{n,K} \tilde{z}_j^{n,K})$, $(\tilde{z}_j^{n,K}, \varphi_l )$,
$b^n(\varphi_m, p^{n,K} \tilde{z}_j^{n,K})$ and
$b^n(\tilde{z}_j^{n,K}, \varphi_l )$.  We take advantage of this fact to compute these matrices in a fast way, by means of an offline/online computation strategy.

\vspace{-0.2cm}
\subsection*{Offline stage}
In the offline stage we compute the element contribution to the coefficients of all matrices appearing in system \eqref{eq:sl} as a function of the two parameters $P$ and $S$, that take values at the nodes of a uniform grid, between minimum and maximum feasible values of these parameters.  That is,

\begin{equation}
\label{malla_ps}
\left\lbrace (P_i,S_j)=( \Delta\, i ,  \Delta\,j ),  \quad \forall\, i,j = 1,2,\ldots M \right\rbrace, \quad \mbox{with } \Delta>0.
\end{equation}

In order to set these values, we consider the piecewise affine finite element functions associated to a uniform partition of $\Omega$ with step $h$.  In practical applications the advection dominates and $P$ takes values larger than 1. Also, taking usual values of diffusion coefficient and $h \simeq \Delta t$, $S$ takes low positive values. Moreover, when we compute the spectral series that determines the coefficients of the system matrices as functions of $P$ and $S$, we observe that these values are nearly constant as $P$ and $S$ approaches $20$. For instance, we can see in figures \ref{matcoef1foto}, \ref{matcoef2} and \ref{matcoef1} how the spectral series for the diagonal coefficient of $A_3$ matrix tend to a constant value as $P$ or $S$ increase to 20. Therefore, in numerical tests, we will consider a step $\Delta=0.02$ and $M=1000$ in \eqref{malla_ps}.

\begin{figure}[h!]
\centering
\includegraphics[scale=0.5]{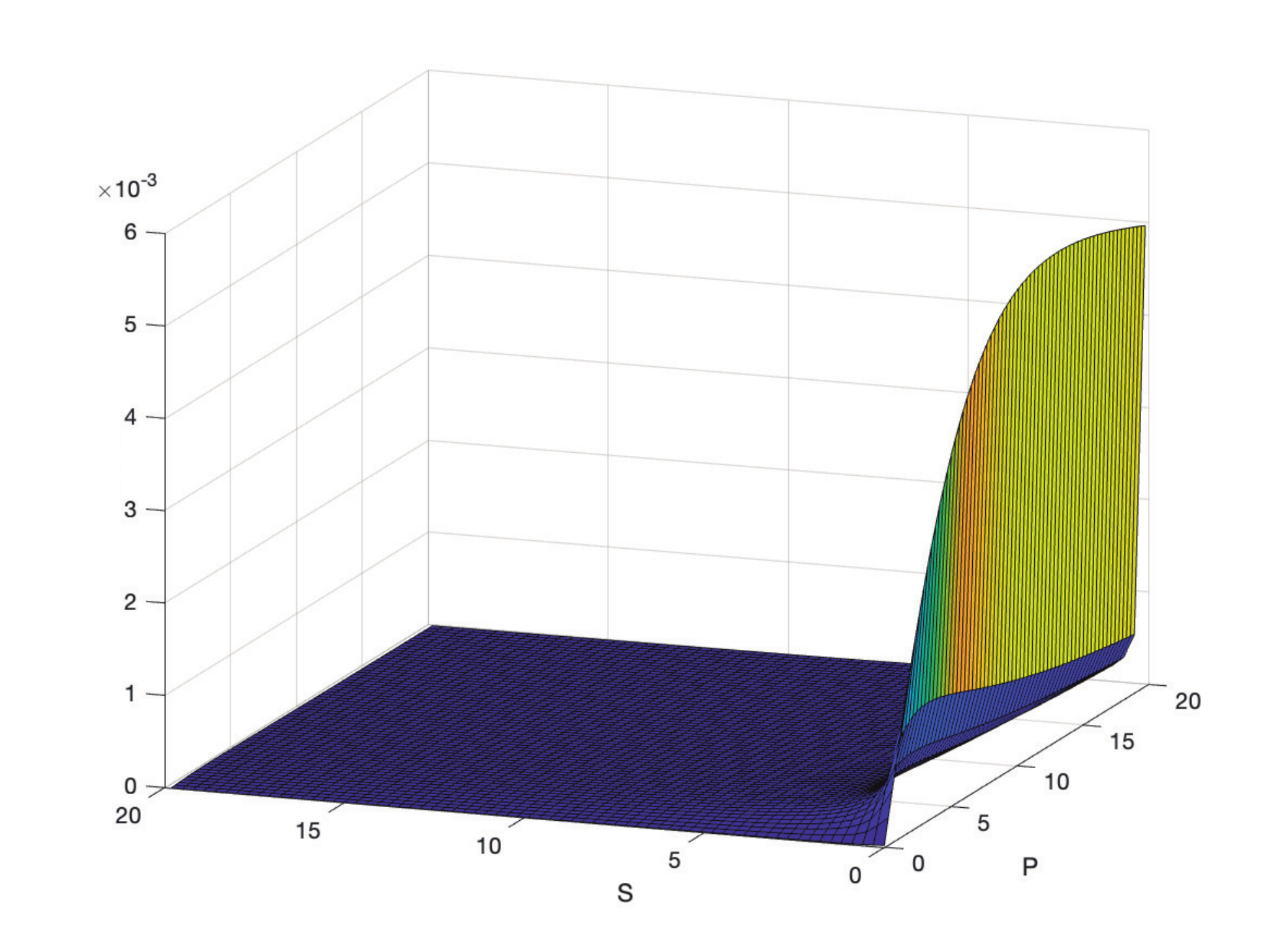}
\caption{Values of the spectral series to compute the diagonal coefficient of matrix $A_3$ for each pair $(P,S)$.}
\label{matcoef1foto}
\end{figure}

\begin{figure}[h]
\begin{minipage}[b]{0.5\linewidth}
\centering
\includegraphics[width=9cm]{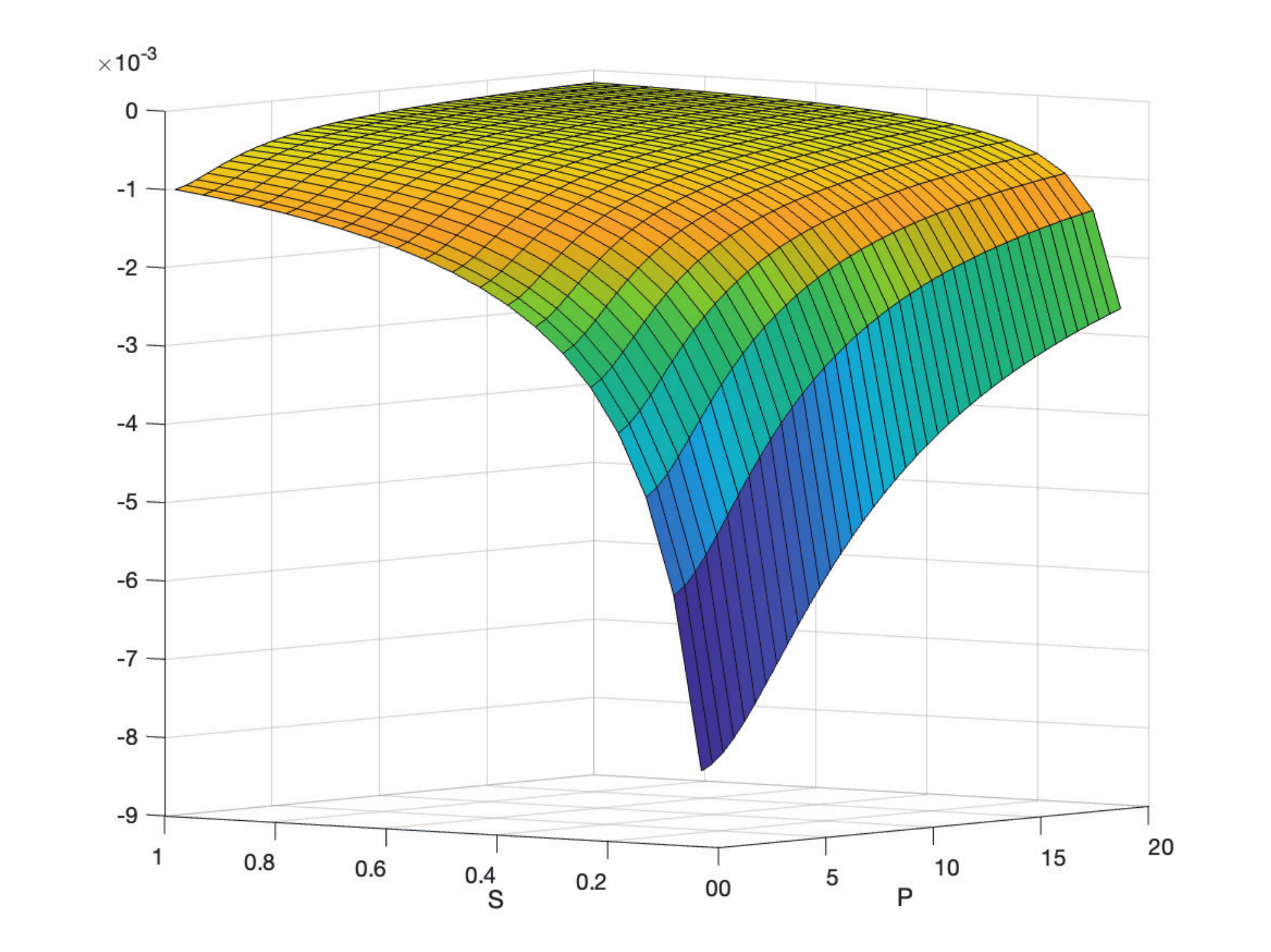}
\caption{Values of the spectral series to compute the diagonal coefficient of matrix $A_3$ for
$(P,S) \in (0,20)\times (0,1)$. }
\label{matcoef2}
\end{minipage}
\hspace{0.5cm}
\begin{minipage}[b]{0.5\linewidth}
\centering
\includegraphics[width=9cm]{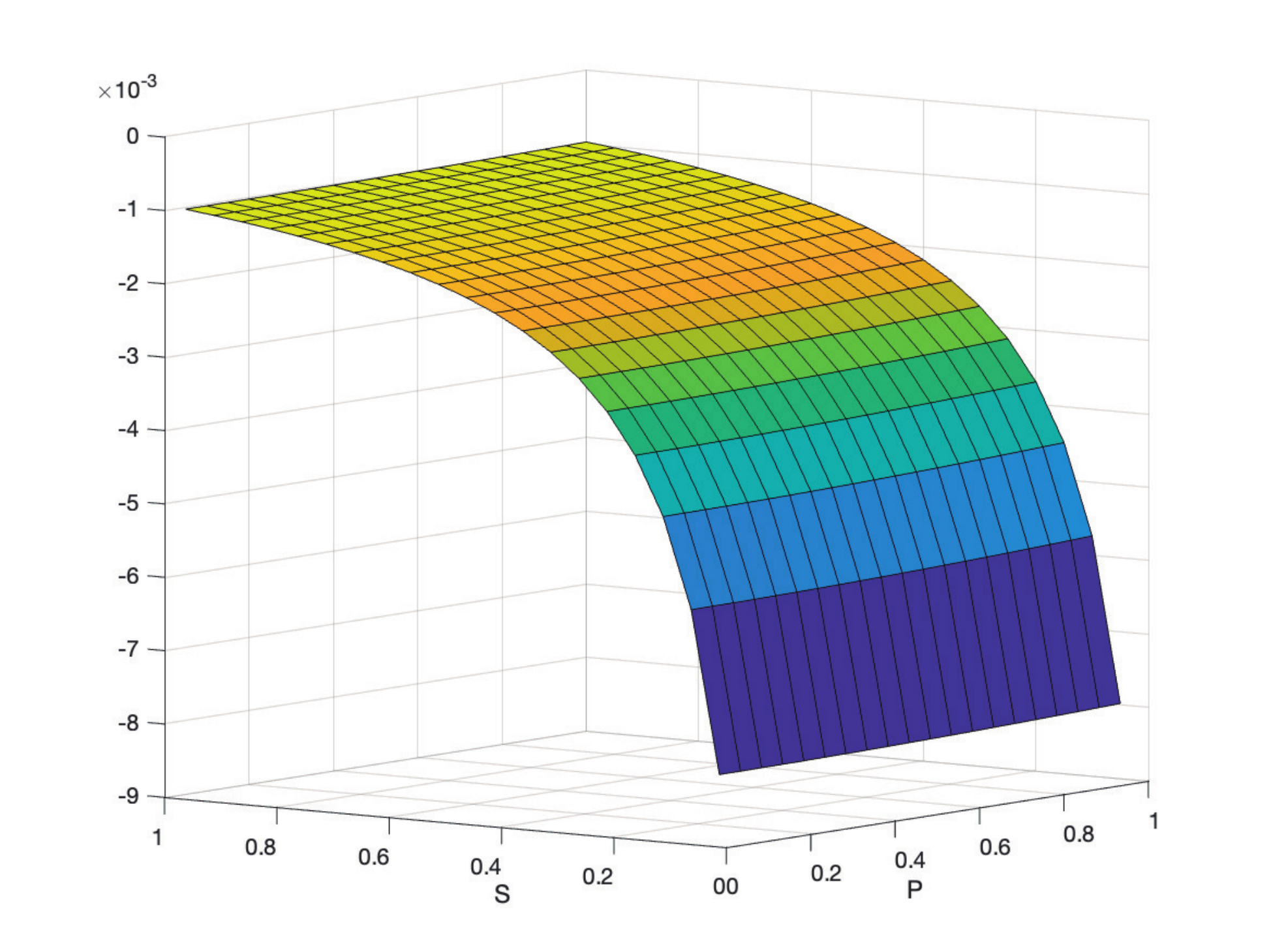}
\caption{Values of the spectral series to compute the diagonal coefficient of matrix $A_3$ for
$(P,S) \in (0,1)\times (0,1)$.}
\label{matcoef1}
\end{minipage}
\end{figure}

To do the computations in this stage, in order to avoid computational roundoff problems due to large velocities,
we express the eigenfunctions of the advection-diffusion operator given in \eqref{eq:zjk} in terms of the midpoint of the grid elements $\d x_{\frac{l,l+1}{2}} = \d \frac{x_l+x_{l+1}}{2}$. That is, we consider
\begin{equation*}
\label{eq:eigen2}
\tilde{z}_j^K=\sqrt{\frac{2}{h_K}} \exp\left(\frac{|a_K|}{2\mu}(x- x_{\frac{l,l+1}{2}})\right)
\sin \left(j\pi \frac{x- x_{\frac{l,l+1}{2}}}{h_K}\right), \quad \mbox{for any } j\in\N.
\end{equation*}
We further truncate the spectral series neglecting all the terms following to the first term that reaches an absolute value less than a prescribed threshold $\varepsilon$. Actually, we have taken $\varepsilon=10^{-10}$.  In Figure \ref{sumautofunM1} we represent the number of these summands needed to reach a first term with absolute value smaller than this $\varepsilon$ for the series defining the diagonal coefficient of $A_3$ matrix. As we can see,  more terms are needed as $P$ increases and as $S$ decreases to $0$.
\begin{figure}[h!]
\centering
\includegraphics[width=11cm]{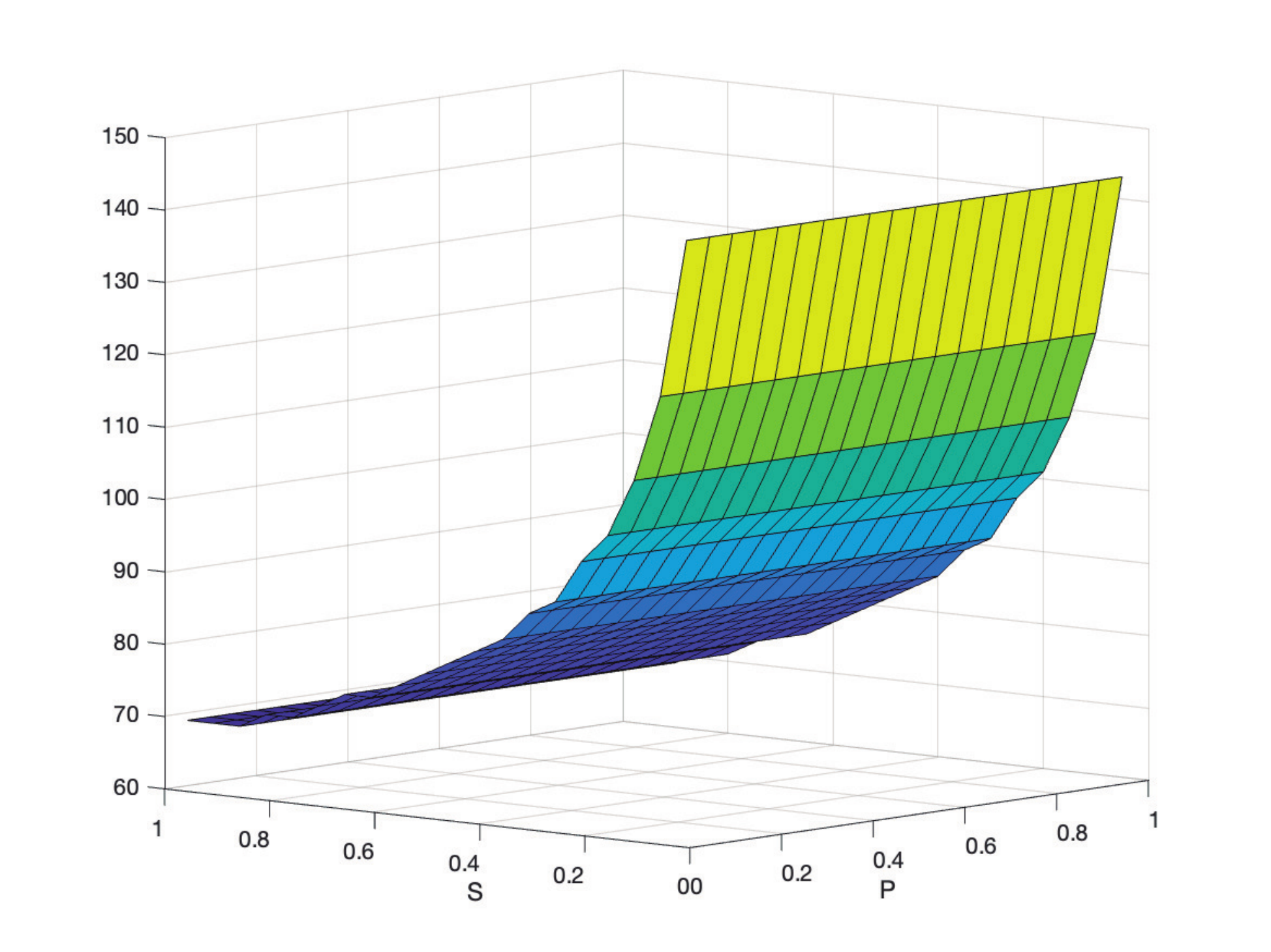}
\caption{Number of summands needed to reach a first term with absolute value lower than $\varepsilon=10^{-10}$
for the series defining the diagonal coefficient of matrix $A_3$, in terms of $(P,S)$.}
\label{sumautofunM1}
\end{figure}

\vspace{-0.2cm}
\subsection*{Online stage}
In the online stage, for each grid element $K$ we compute the contribution of this element to the
coefficients of all matrices appearing in system \eqref{eq:sl}. Then, we sum up over grid elements, to calculate these coefficients.

For that, we determine $P_K$ and $S_K$ and find the indices
$i,\, j \in {1,\ldots, M}$ such that $(P_K,\, S_K)$ belongs to $[P_i , P_{i+1}] \times [S_j , S_{j+1}]$. In other case,
if $P_k<\Delta$ we set $i=1$ and if $P_K >\Delta M$ we set $i=M-1$, and similarly for $j$ in terms of $S_K$.

As we see above, each element contribution is a function of $P_K$ and $S_K$ that we denote $C(P_K,S_K)$ in a generic way. For instance, for matrix $A^n_1$,
$$
 C(P_K,S_K)= \sum_{j=1}^{\infty} \beta_j^{n,K}
(\varphi_m, p^{n,K} \tilde{z}_j^{n,K})(\tilde{z}_j^{n,K}, \varphi_l ).
$$

Then, we compute $C(P_k,S_K)$ by the following second-order interpolation formula:
$$
C(P_K,S_K) \simeq \sum_{k=1}^4 \frac{Q_k}{Q} \, C(\alpha_k),
$$
where the $\alpha_k$ are the four corners of the cell $[P_i , P_{i+1}] \times [S_j, S_{j+1}]$, \break $Q=\Delta^2$ is its area and the $Q_k$ are the areas of the four rectangles in which the cell is split by $(P_k,S_K)$ (see Figure \ref{cuadrado2t}).

\begin{figure}[h]
\centering
\includegraphics[width=6cm]{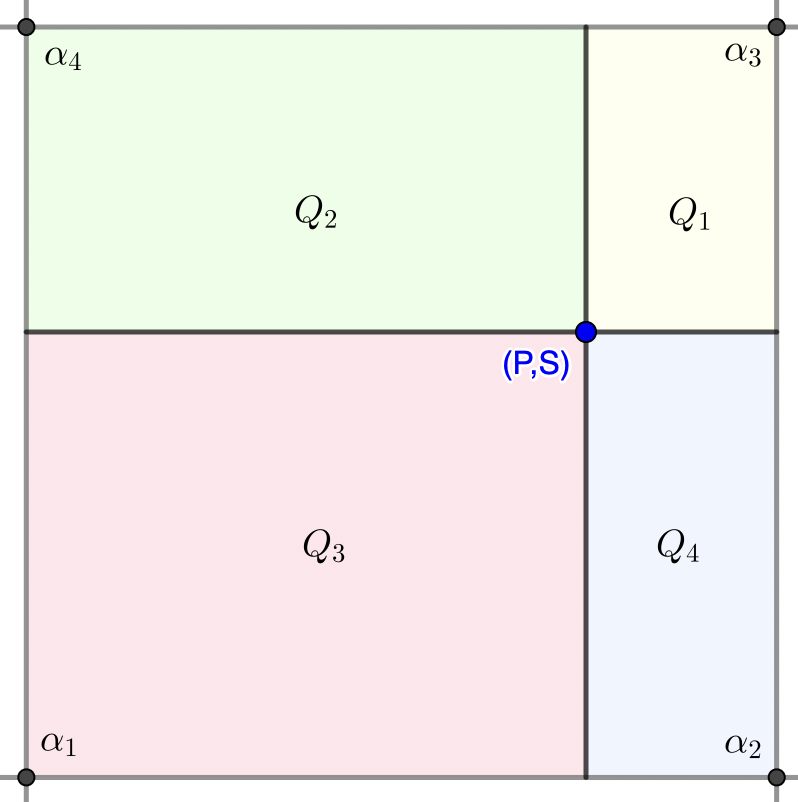}
\caption{Splitting of interpolation cell for online computation of matrices coefficients.}
\label{cuadrado2t}
\end{figure}

\section{Numerical Tests}
\label{se:nr}
In this section, we present the numerical results obtained with the spectral method to solve 1D advection-diffusion problems. Our purpose, on the one hand, is to confirm the theoretical results stated in Corollary \ref{corol} for the spectral VMS method and, on the other hand, test the accuracy of the spectral VMS and feasible spectral VMS methods for problems with strong advection-dominance, in particular by comparison with several stabilised methods.

\subsection{Test 1: Accuracy of spectral VMS method   for constant advection velocity}
To test the property stated in Corollary \ref{corol}, we consider the following advection-diffusion problem:\begin{equation}
\label{eq:cd1}
\left\{\begin{array}{ll}
\partial_t u +a \, \partial_x u-\mu \, \partial^2_{xx} u = 0 &
\mbox{in }(0,1)\times(0,T),
\\[0,2cm]
u(0,t)=\exp((\mu-a)t),\quad u(1,t)=\exp(1+(\mu-a)t) & \mbox{on }(0,T), \\[0,2cm]
u(x,0)=\exp(x) & \mbox{on } (0,1),
\end{array}\right.
\end{equation}
whose exact solution is given by 
$
\exp(x+(\mu-a)t).
$

We set $T=0.1$, $a=1$ and $\mu=20$.  We apply the spectral VMS method \eqref{eq:SVMS} to solve this problem with time step $\Delta t=0.01$ and piecewise affine finite element space on a uniform partition of interval $(0,1)$ with steps $h=0.05/(2^i)$ for $i=2,3,...,7$. We have truncated the spectral expansions that yield the small scales $\tilde{u}_h^n$ to $10$ eigenfunctions.  The errors in $l^\infty(L^2)$ and $l^2(H^1)$ norms computed at grid nodes are represented in Figure \ref{fig_test0}. We observe that, indeed, the errors quite closely do not depend on the space step $h$.

Moreover, we have computed the convergence orders in time, obtaining very closely order 1 in  $l^2(H^1)$ norm and order 2 in $l^\infty(L^2)$ norm, as could be expected.

\begin{figure}[h!]
\centering
\includegraphics[width=10cm]{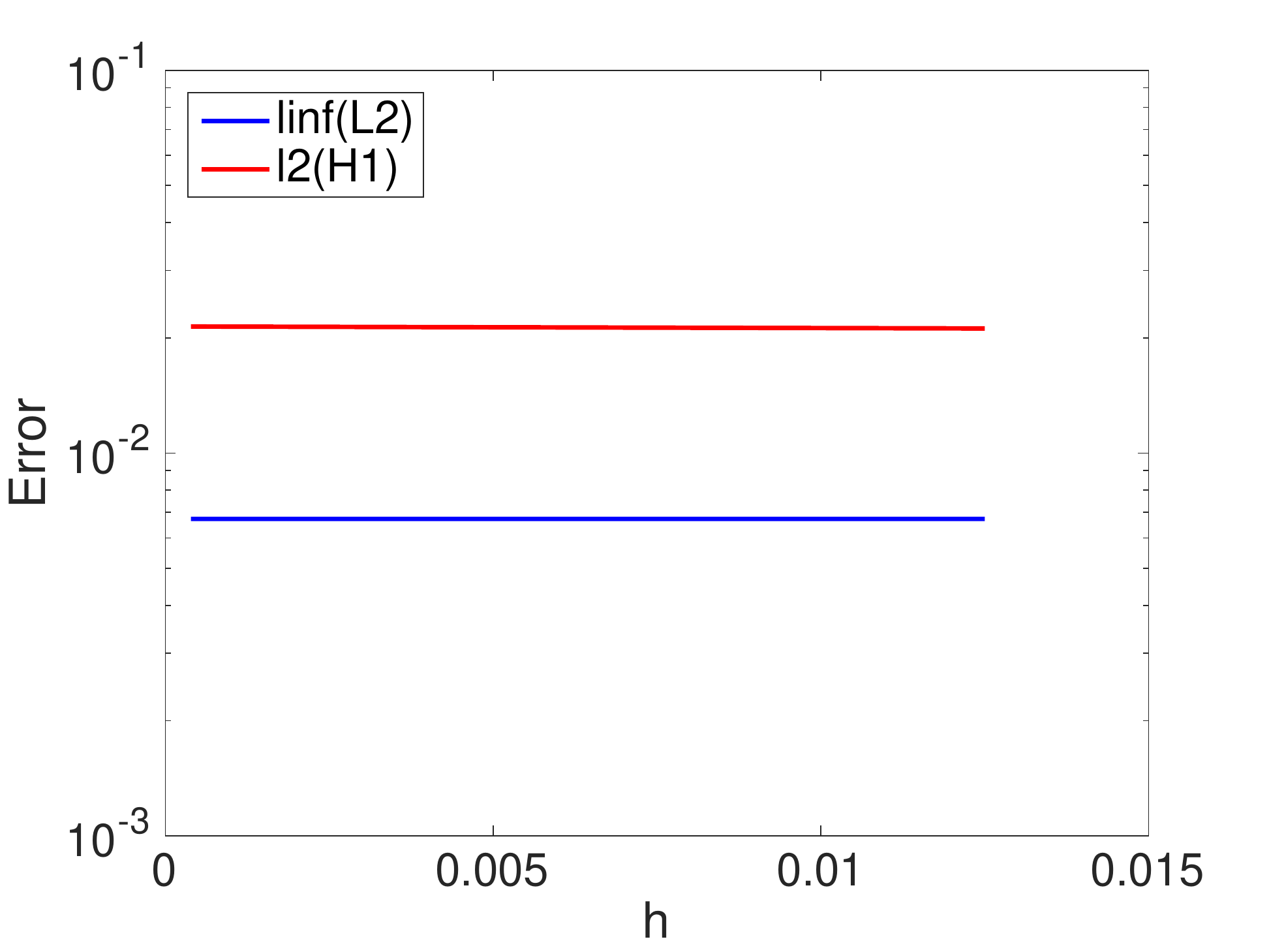}
\caption{Test 1. $l^\infty(L^2)$ and $l^2(H^1)$ errors for the spectral VMS solution of problem \eqref{eq:cd1}.}
\label{fig_test0}
\end{figure}

\vspace{1cm}
In the following numerical experiments we consider the 1D problem (\ref{eq:cd}) setting $\Omega=(0,1)$, with constant velocity field $a$, source term $f=0$ and the hat-shaped initial condition
\begin{equation}
\label{solexacta}
u_0 = \left\{ \begin{array}{ll} 1 & \text{if } |x-0.45|\leq 0.25 ,  \\ 0  & \text{otherwise.} \end{array} \right.
\end{equation}
We also set $X_h$ to be the piecewise affine finite element space constructed on a uniform partition of interval $(0,1)$ with step size $h$.

\subsection{Test 2: Accuracy of spectral VMS method}
\label{nse:SVMS}

\vspace{0.2cm}
{\bf Very large P\'eclet numbers}

\noindent In this test we examine the accuracy of spectral VMS method for very high P\'eclet numbers.
To do that, we set $a=1000$ and $\mu = 1$,  and solve this problem by the spectral VMS method \eqref{eq:SVMS}, truncating to 150 spectral basis functions the series \eqref{eq:smalls_d} that yield the sub-grid components.
The solution interacts with the boundary condition at $x=1$ in times of order $1/a$, that is, $10^{-3}$. We then set a time-step $\Delta t = 10^{-3}$.  Moreover, we set $h=0.02$ that corresponds to $P=10$ and $S=2.5$. We present the results obtained in Figure \ref{fig_test1}, where we represent the Galerkin solution (in red) on the left panels  and the spectral solution (in cyan) on the  right panels, both with the exact solution (in blue): in $(a)$ the first 4 time-steps,  in $(b)$ time-steps from 5 to 7 and in $c$ times-steps 8 and 9.   By Corollary \ref{corol} the discrete solution coincides at the grid nodes with the exact solution of the implicit Euler semi-discretisation, the expected errors at grid nodes are of order $\Delta t=10^{-3}$. We can see that the spectral solution indeed is very close to the exact solution at grid nodes.
\begin{figure}[h]
$(a)$ \hspace{8cm} $(b)$
\centering
\includegraphics[width=8.42cm]{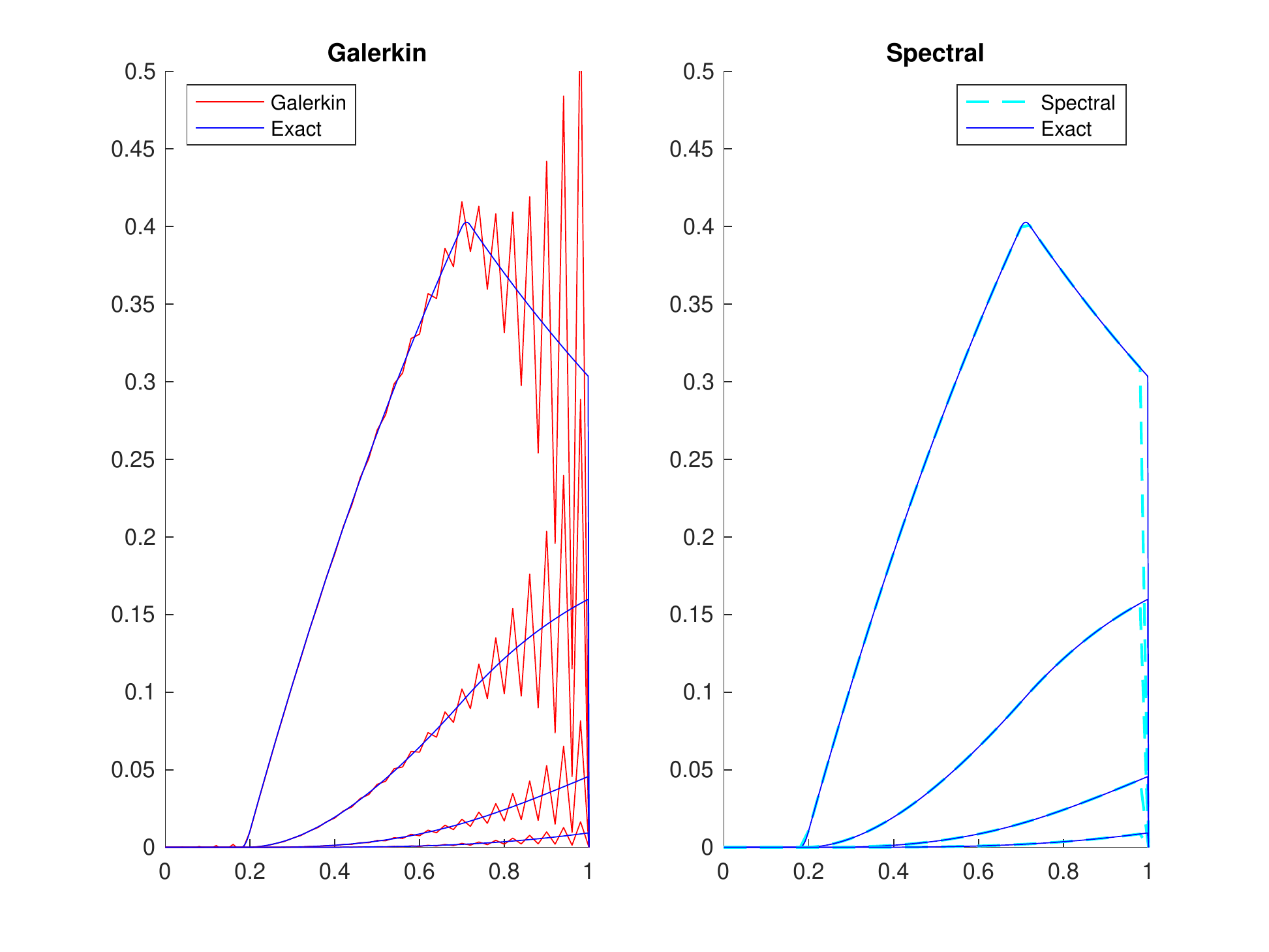}
\hfill
\includegraphics[width=8.42cm]{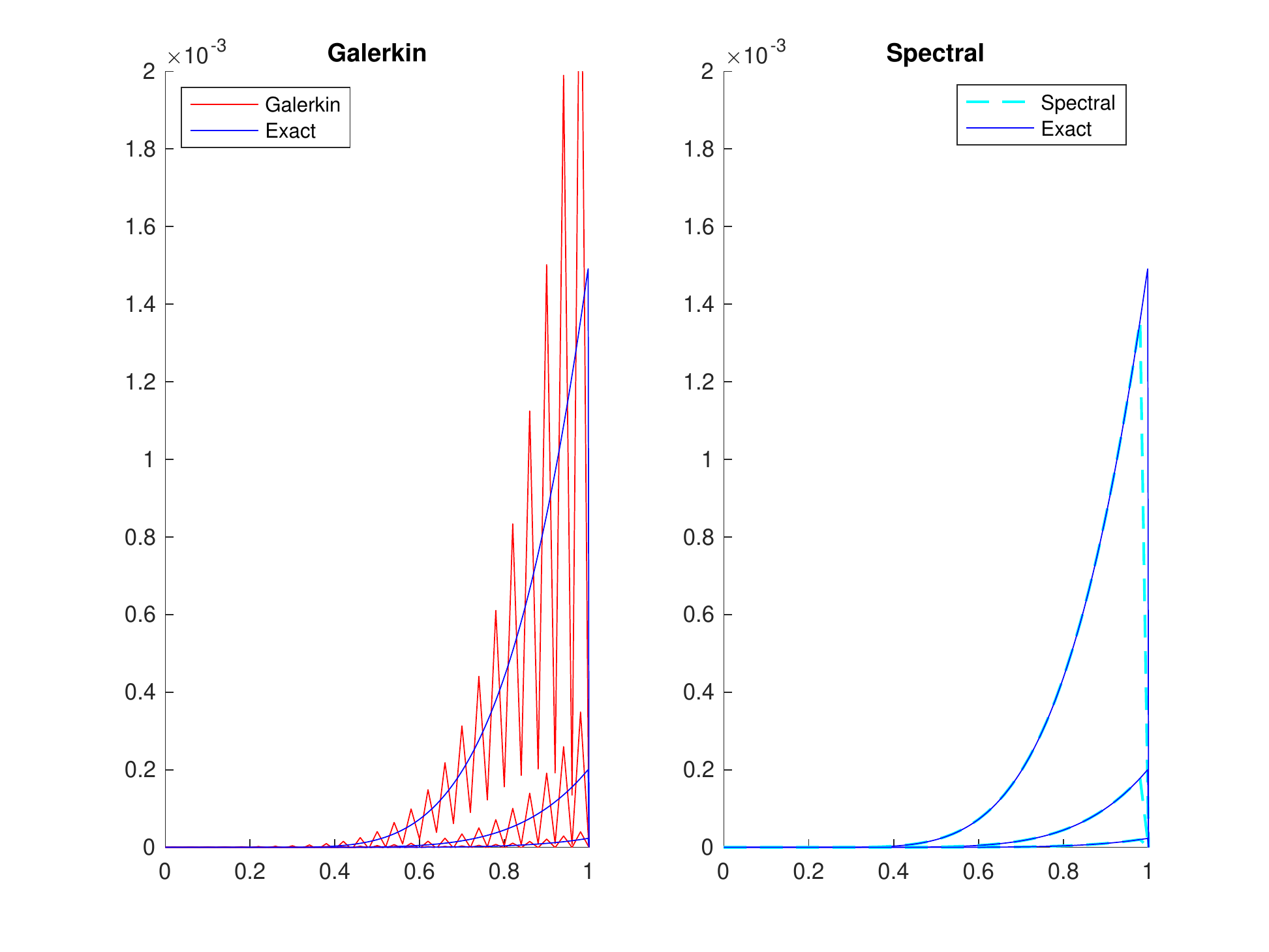} $(c)$
\vfill
\includegraphics[width=8.42cm]{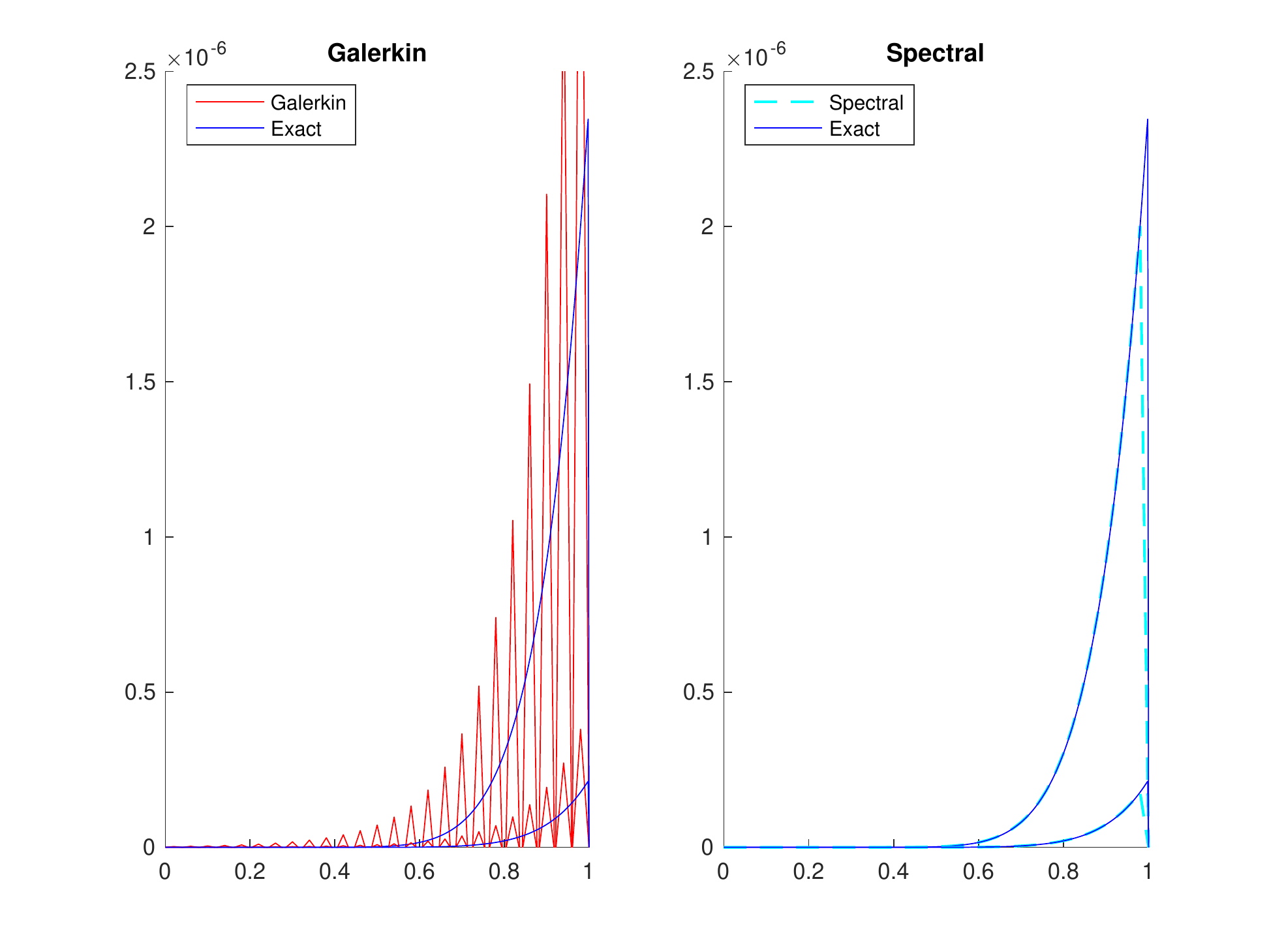}
\caption{Solution of problem \eqref{eq:cd} for $a=1000,  \mu = 1,  f=0$ and $u_0$ given by (\ref{solexacta}) with
$\Delta t=10^{-3}$ and $h=0.02$ ($P = 10$, $S = 2.5$). The spectral VMS solution is compared to the exact solution and the Galerkin solution.  The results for time-steps numbers 1 to 4, 5 to 7 and 8 to 9 are respectively represented in figures $(a)$, $(b)$ and $(c)$.}
\label{fig_test1}
\end{figure}

%
As the discrete solution coincides at the grid nodes with the exact solution of the implicit Euler semi-discretisation and $u^0$ is exact, then $u^1_h$ should coincide with the exact solution at grid nodes. This can already be observed in Figure \ref{fig_test1} $(a)$. We also test this result with different discretisation parameters. We actually set $\Delta t = 10^{-5}$ and $h=0.02$ that corresponds to $P=10$ and $S=0.025$. The solution in the first time-step is represented in Figure \ref{fig_test2} $(a)$ and a zoom around $x=0.7$ in depicted in $(b)$.  Indeed the discrete solutions coincides with the exact one at grid nodes.
\begin{figure}[h!]
$(a)$ \hspace{8cm} $(b)$
\centering
\includegraphics[width=8.42cm]{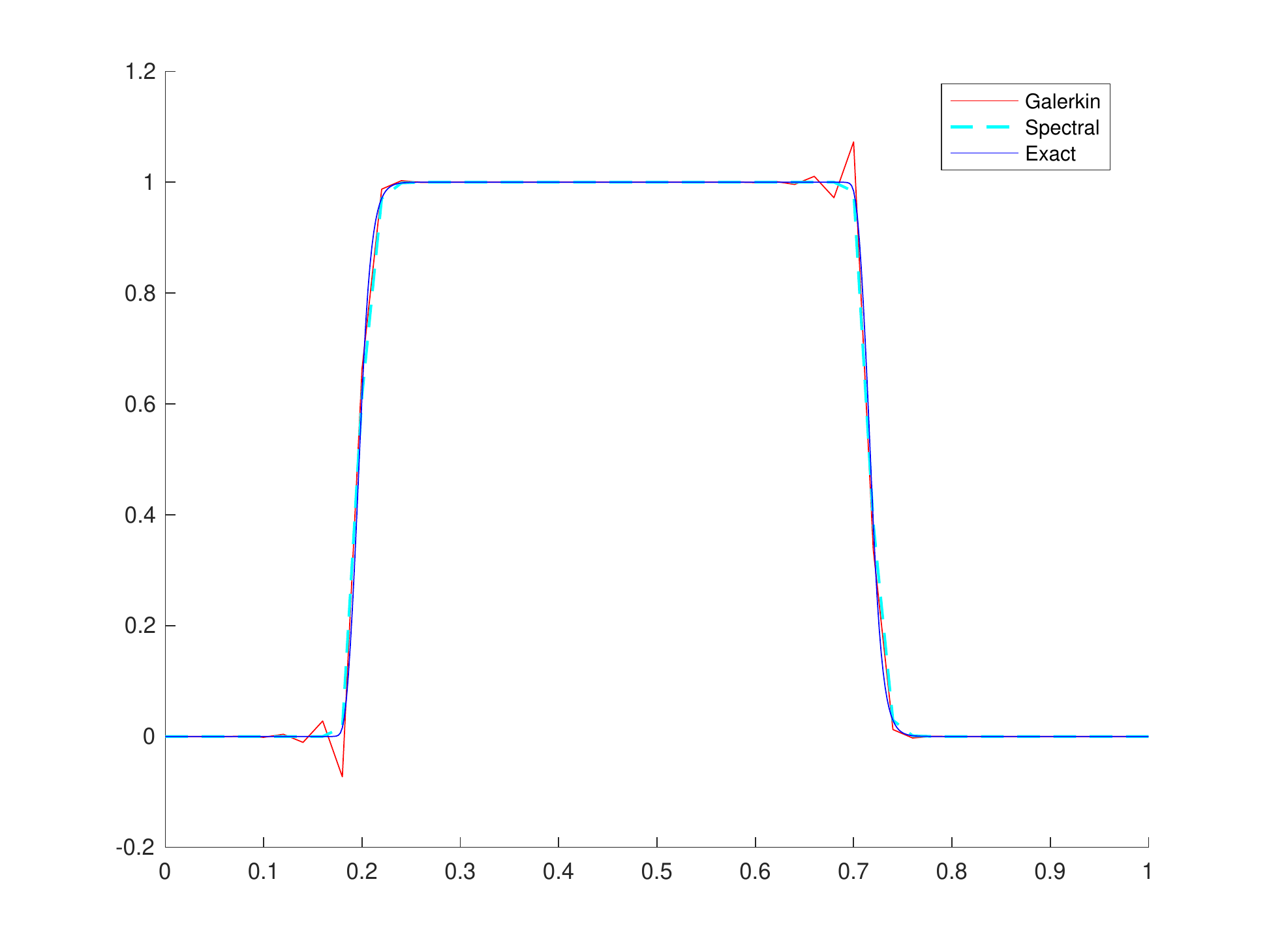}
\hfill
\includegraphics[width=8.42cm]{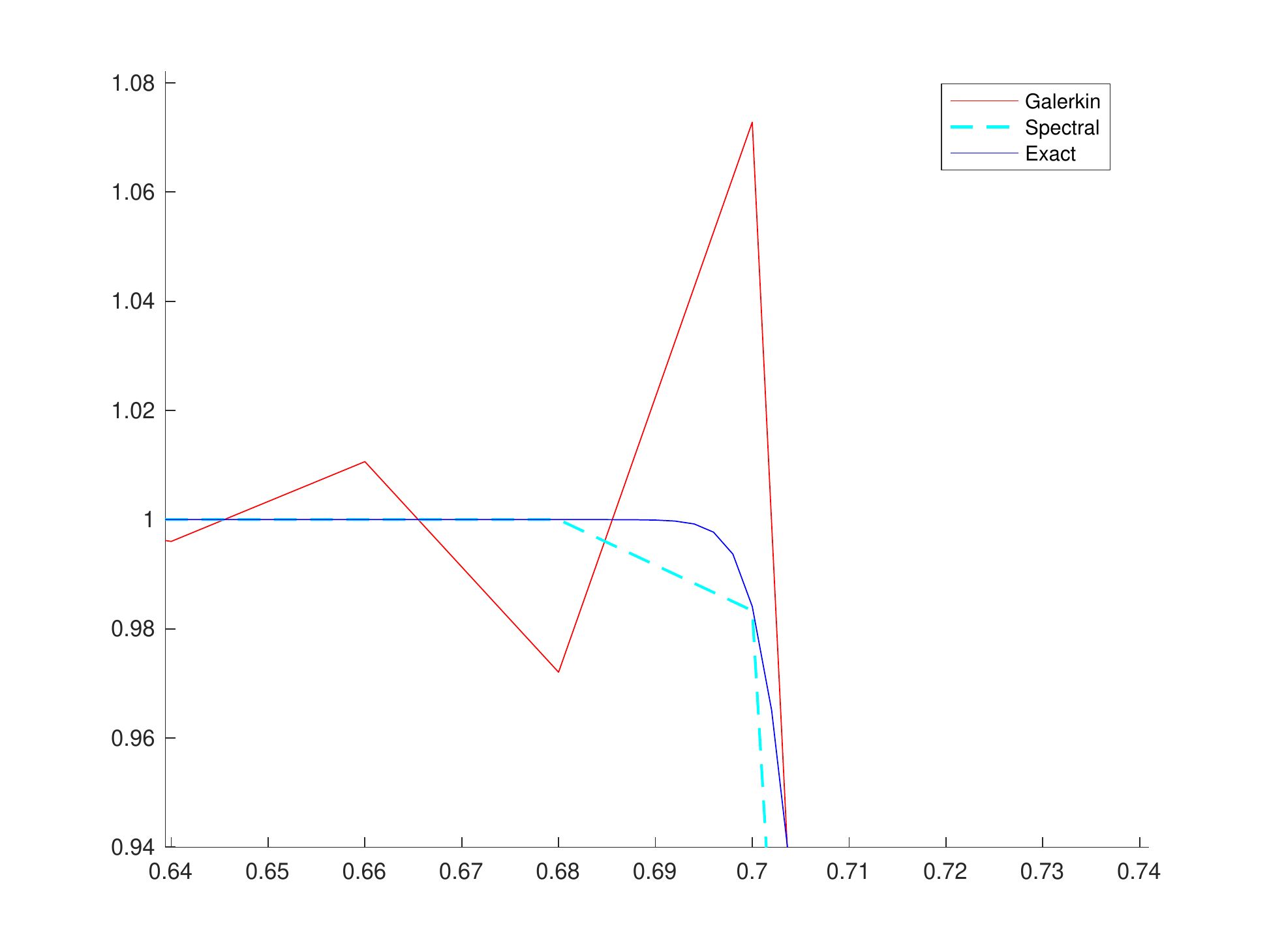}
\caption{Solution of problem \eqref{eq:cd} for $a=1000,  \mu = 1,  f=0$ and $u_0$ given by (\ref{solexacta})
with $\Delta t=10^{-5}$ and $h=0.02$ ($P=10$, $S=0.025$).  The spectral VMS solution is compared to the exact solution and the Galerkin solution at first time step.  Figures $(a)$ and $(b)$ respectively show these solutions in the whole domain and a zoom around $x=0.7$. }
\label{fig_test2}
\end{figure}

\vspace{0.2cm}
\noindent {\bf Very small time steps}

\noindent We test here the arising of spurious oscillations due to extra small time-steps. These spu\-rious oscillations occur in the solutions provided by the Galerkin discretisation when $CFL <CFL_{bound}  = P/(3(1-P))$ (see \cite{hauke}). For that, we consider the same problem as in this section but with $a=20$,  $h=0.01$ and the time-step $\Delta t$ is chosen such that $CFL/CFL_{bound}= 1/2$.  We obtain the results shown in Figure \ref{fig_test3}, where we have represented the first five time-steps.  As one can see the spectral solution does not present any oscillation.
\begin{figure}[h!]
\centering
\includegraphics[width=10cm]{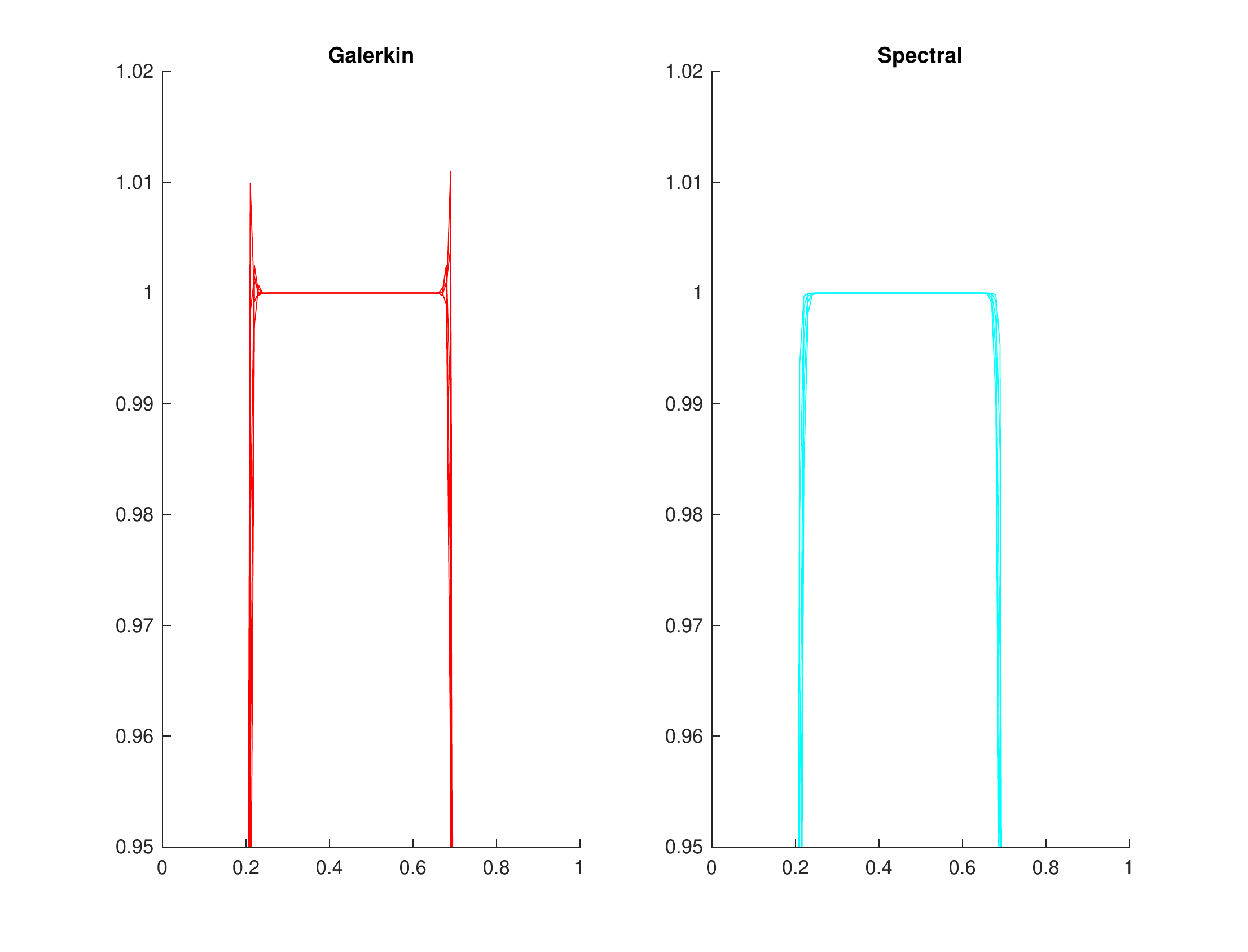}
\caption{Solution of problem \eqref{eq:cd} for $a=20,  \mu = 1,  f=0$ and $u_0$ given by (\ref{solexacta})
with $h=0.01$ and $\Delta t$ such that $CFL/CFL_{bound}= 1/2$ ($P=0.1$, $S=0.0926$).  Red lines represent Galerkin solution and cyan lines represent spectral VMS solution  in each step-time.}
\label{fig_test3}
\end{figure}

\subsection{Test 3: Accuracy of the feasible spectral VMS method. \\ Comparison with other stabilised methods}
We next proceed to compare the results obtained with the feasible spectral VMS method \eqref{eq:SVMS_F} with those obtained by several stabilised methods.

Stabilised methods add specific stabilising terms to the Galerkin discretisation, generating the following matrix scheme,
\begin{equation*}
(M+\Delta t \, R^n + \Delta t \, a^2 \, \tau \, M_{s}) \, \mathbf{u}^{n+1} = M \, \mathbf{u}^n,
\end{equation*}
where $M$ and $R^n$ are, respectively, mass and stiffness matrices, while $M_{s}$ is a tridiagonal matrix defined by $({M_s})_{i,i} = \frac{2}{h}, \quad ({M_s})_{i+1,i}=({M_s})_{i,i+1} = -\frac{1}{h}$. Each  stabilised method is determined by the stabilised coefficient $\tau$. In particular,  we consider:
\begin{enumerate}
\item The optimal stabilisation coefficient for 1D steady advection-diffusion equation \cite{christie,johnnovo},
\begin{equation*}
\tau_{1D} = \displaystyle\frac{\mu}{|a|^2} (P\coth(P)-1).
\end{equation*}
\item The stabilisation coefficient based on orthogonal sub-scales  proposed by Codina in \cite{codina12},
\begin{equation*}
\tau_{C}= \displaystyle \left(  \left( 4 \frac{\mu}{h^2} \right)^2 + \left( 2 \frac{|a|}{h} \right)^2 \right)^{-1/2}.
\end{equation*}
\item The stabilisation coefficient based on $L_2$ proposed by Hauke et. al. in  \cite{haukee},
\begin{equation*}
\tau_H = \displaystyle \min \left\{ \frac{h}{\sqrt{3}|a|} ,  \frac{h^2}{24.24\mu} , \Delta t \right\}.
\end{equation*}
\item The stabilisation coefficient separating the diffusion-dominated from the convection-dominated regimes proposed by Franca in \cite{solmactom2},
\begin{equation*}
\tau_{F}= \d\frac{h}{|a|}\, \min\{P,\tilde{P}\},
\end{equation*}
where $\tilde{P}>0$ is a threshold separating the diffusion dominated ($P \le \tilde{P}$) to the advection dominated
($Pe > \tilde{P}$) regimes.
\end{enumerate}
In figures \ref{fig_comp1}, \ref{fig_comp2} and \ref{fig_comp3},  we show the solutions of each method for different values of $P$ and $S$, always for advection-dominated regime $P>1$. We also display the errors in $l^{\infty}(L^2)$ and $l^2(H^1)$ norms for the solutions of these problems in tables 1, 2 and 3. As it can be observed in the three tables,  spectral method reduces the error between 10 and 100 times compared to the stabilised methods, without presenting oscillations.

\begin{figure}[h!]
\centering
\includegraphics[width=9.8cm]{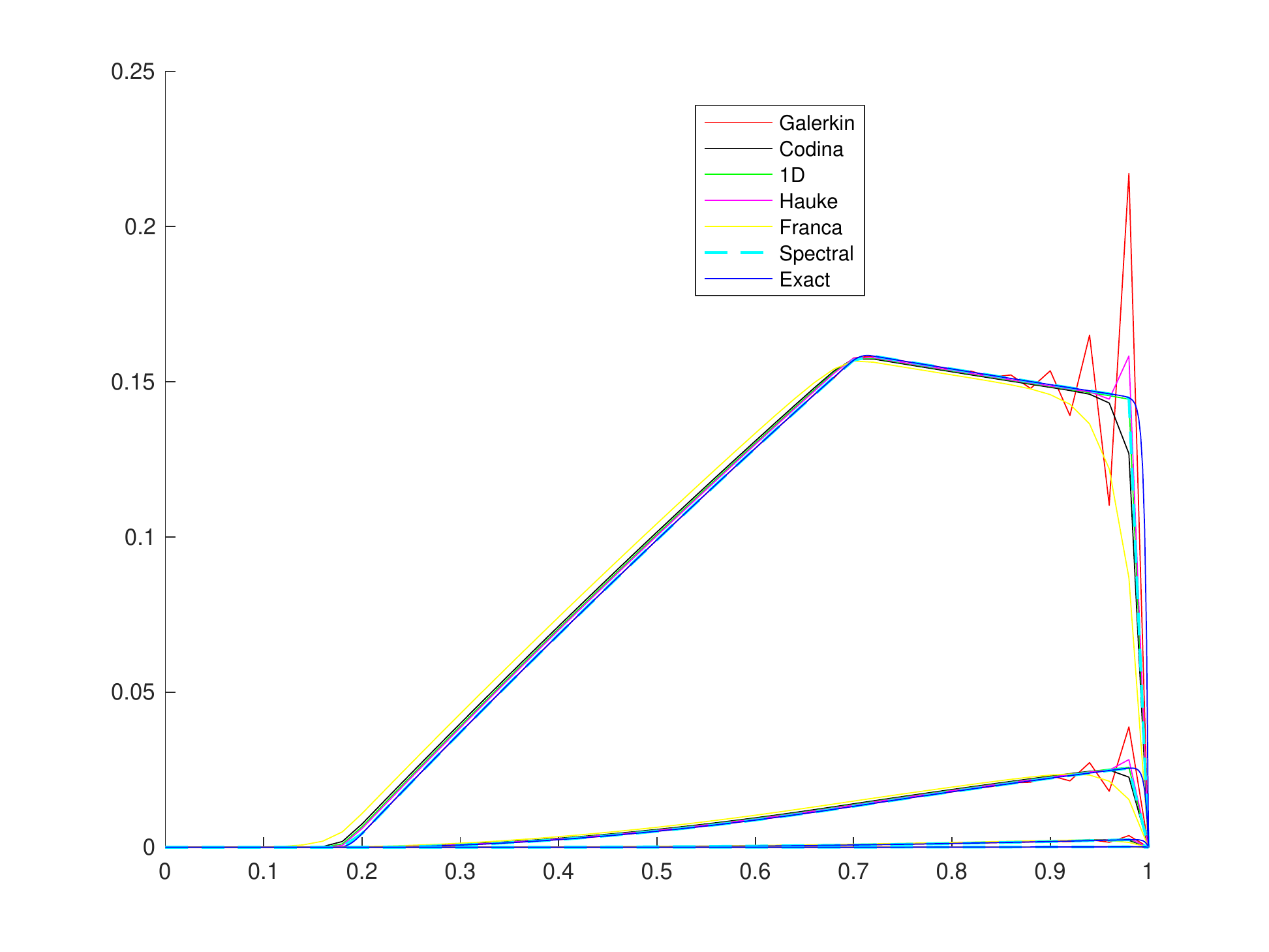}
\caption{Comparison of different stabilised methods to solve problem \eqref{eq:cd} when $P=3$, $S=25$ with $\Delta t=10^{-2}$ and $h=0.02$. Solutions in the three first time-steps. }
\label{fig_comp1}
\end{figure}
\begin{table}[h!]
\centering
\begin{tabular}{|l|l|r|r|}
\hline
& Method    & \multicolumn{1}{l|}{$l^{\infty}(L^2)$} & \multicolumn{1}{l|}{$l^2(H^1)$} \\ \hline \hline
\multirow{6}{*}{\begin{tabular}[c]{@{}l@{}}$P=3$\\[0.2cm] $S=25$\end{tabular}}
& Galerkin &   1.1784e-02   & 4.7505e-02              \\ \cline{2-4}
& Spectral & 8.7889e-06          & 5.4716e-05              \\ \cline{2-4}
& Codina    & 3.2285e-03         & 1.4329e-02              \\ \cline{2-4}
& 1D & 1.3805e-03        & 1.3446e-03               \\ \cline{2-4}
& Hauke     & 2.1713e-03         & 1.1124e-02            \\ \cline{2-4}
& Franca   & 9.9020e-03       & 5.0380e-02            \\ \hline
\end{tabular}
\caption{$l^{\infty}(L^2)$ and $l^2(H^1)$ errors for the solutions represented in Figure \ref{fig_comp1}. }
\end{table}
\begin{figure}[h!]
\centering
\includegraphics[width=8.42cm]{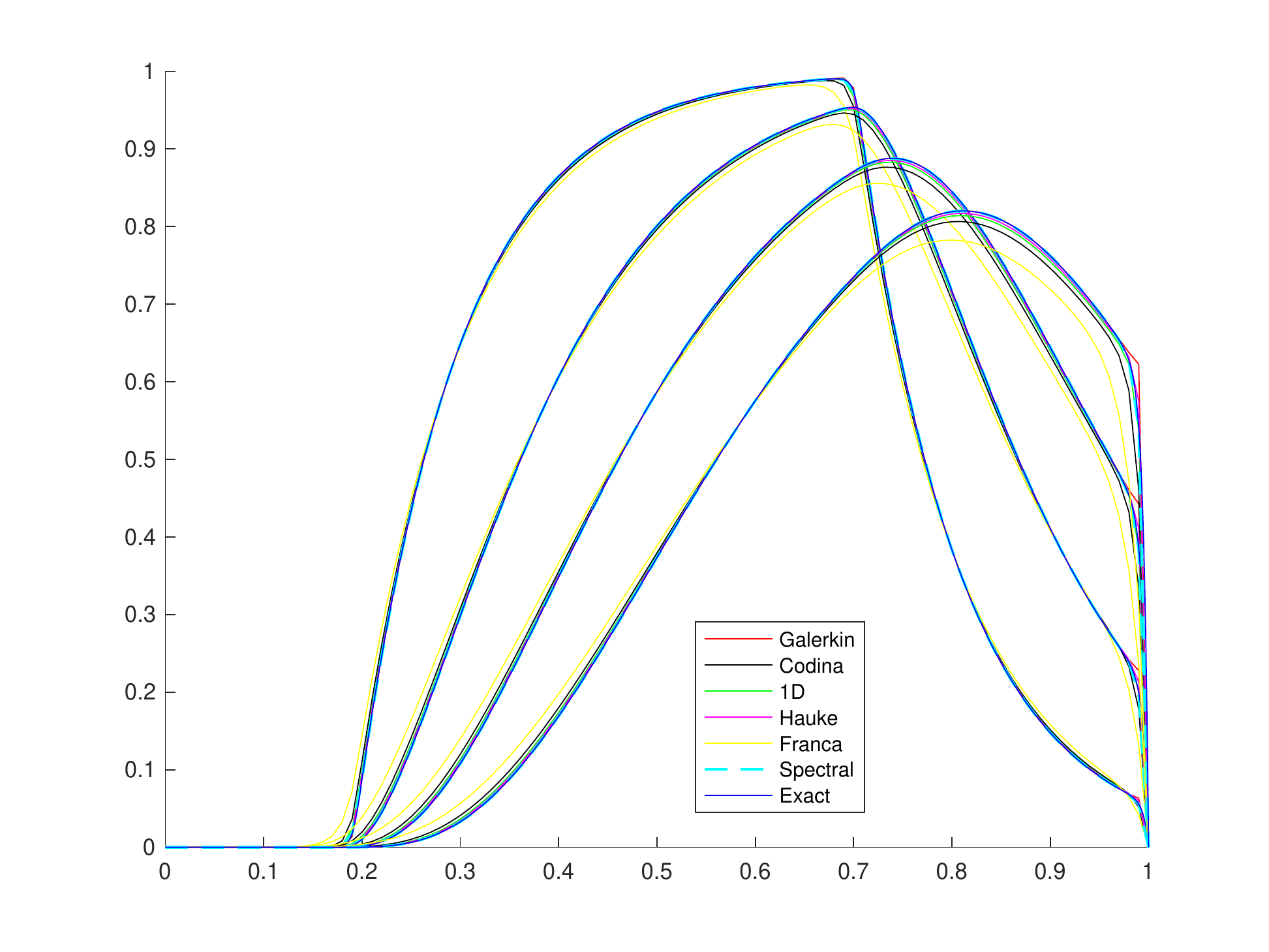}
\hfill
\includegraphics[width=8.42cm]{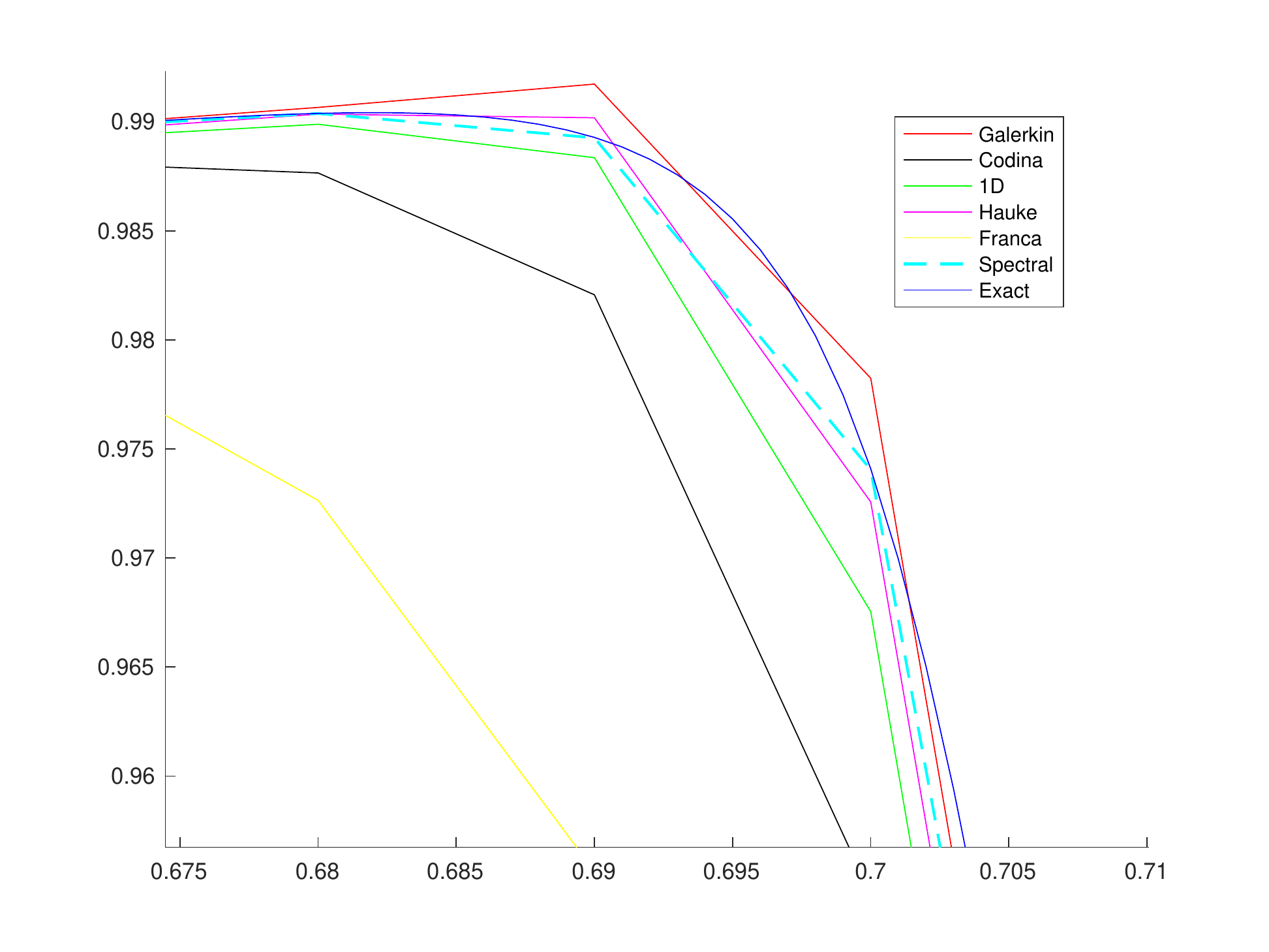}
\caption{Comparison of different stabilised methods to solve problem \eqref{eq:cd} when $P=1$ and $S=5$ with $\Delta t=10^{-3}$ and $h=10^{-2}$. Solutions in the three first time-steps. Right: zoom around $x=0.7$. }
\label{fig_comp2}
\end{figure}
\begin{table}[h]
\centering
\begin{tabular}{|l|l|r|r|}
\hline
& Method    & \multicolumn{1}{l|}{$l^{\infty}(L^2)$} & \multicolumn{1}{l|}{$l^2(H^1)$} \\ \hline \hline
\multirow{6}{*}{\begin{tabular}[c]{@{}l@{}}$P=1$\\[0.2cm] $S=5$\end{tabular}}
& Galerkin &        9.6551e-03    & 7.7424e-02               \\ \cline{2-4}
& Spectral & 7.2887e-05          & 5.2396e-04               \\ \cline{2-4}
& Codina    & 1.3580e-02         & 6.4992e-02               \\ \cline{2-4}
& 1D & 3.7524e-03        & 5.3902e-03               \\ \cline{2-4}
& Hauke     & 4.2353e-03          & 3.3330e-02            \\ \cline{2-4}
& Franca   & 4.4200e-02      & 3.1419e-01            \\ \hline
\end{tabular}
\caption{$l^{\infty}(L^2)$ and $l^2(H^1)$  errors for the solutions represented in Figure \ref{fig_comp2}. }
\end{table}
\begin{figure}[h!]
\centering
\includegraphics[width=8.42cm]{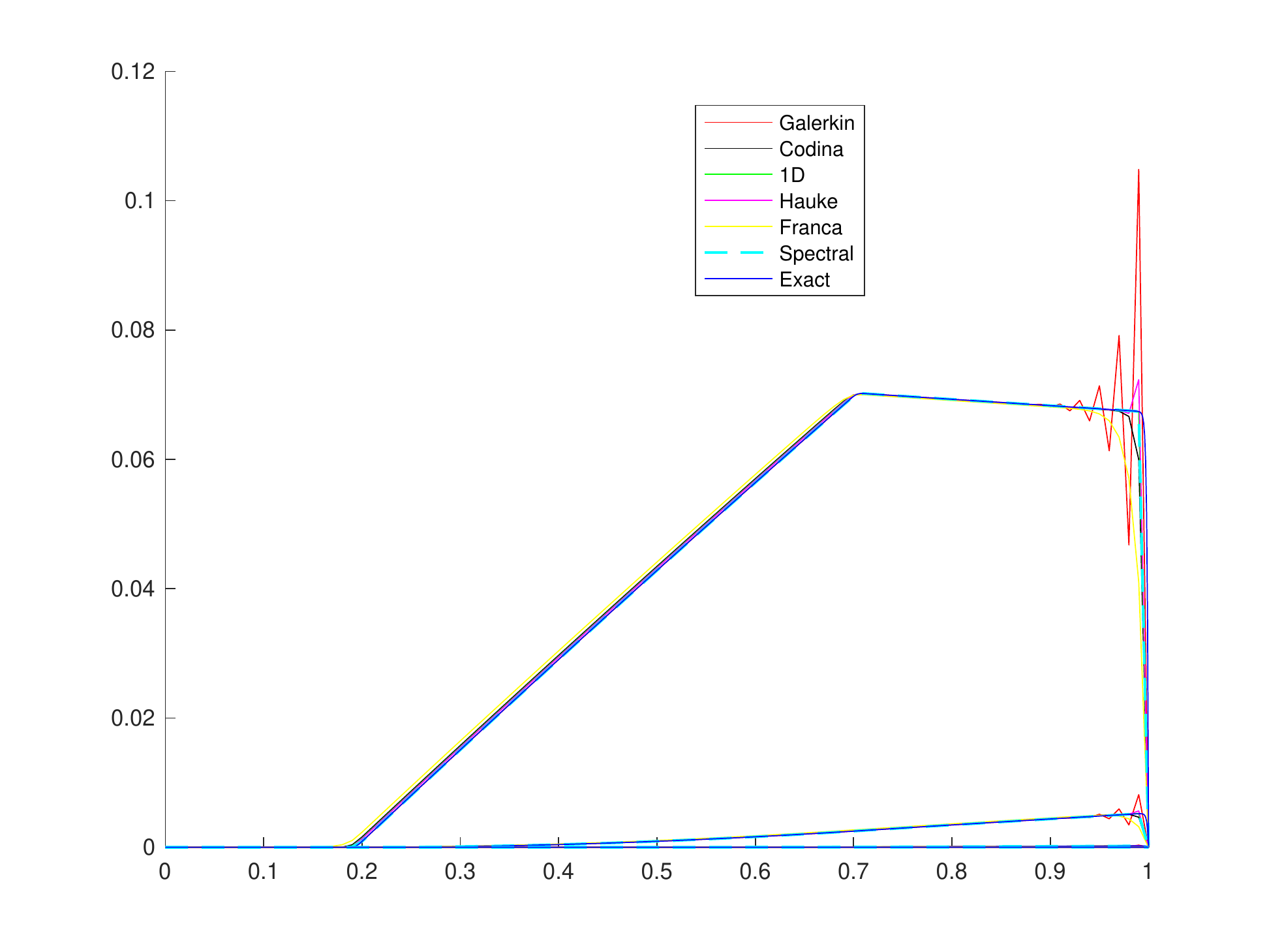}
\hfill
\includegraphics[width=8.42cm]{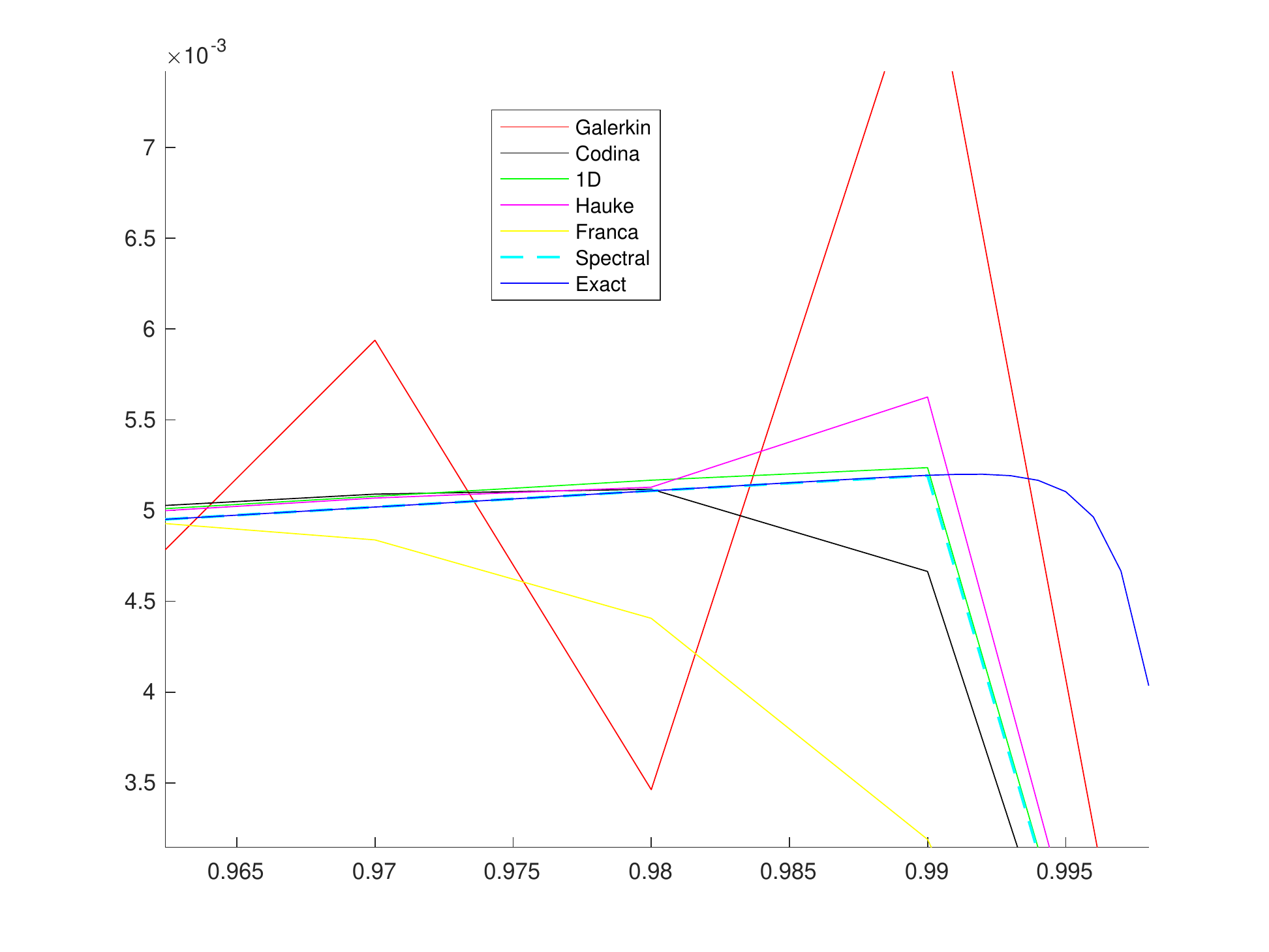}
\caption{Comparison of different stabilised methods to solve problem \eqref{eq:cd} when $P=3.5$ and $S=100$ with $\Delta t = 10^{-2}$ and $h=10^{-2}$.  Solutions in the three first time-steps. Right: zoom around $x=0.99$.}
\label{fig_comp3}
\end{figure}
\begin{table}[h!]
\centering
\begin{tabular}{|l|l|r|r|}
\hline
& Method    & \multicolumn{1}{l|}{$l^{\infty}(L^2)$} & \multicolumn{1}{l|}{$l^2(H^1)$} \\ \hline \hline
\multirow{6}{*}{\begin{tabular}[c]{@{}l@{}}$P=3.5$\\[0.2cm]$S=100$\end{tabular}}
& Galerkin &        4.5006e-03    & 3.3305e-02               \\ \cline{2-4}
& Spectral & 1.6381e-06          & 2.0138e-05               \\ \cline{2-4}
& Codina    & 8.752e-04        & 8.4455e-03               \\ \cline{2-4}
& 1D & 3.3968e-04        & 4.5556e-04               \\ \cline{2-4}
& Hauke     & 5.6656e-04        & 5.6930e-03            \\ \cline{2-4}
& Franca   & 3.0238e-03      & 3.0336e-02            \\ \hline
\end{tabular}
\caption{$l^{\infty}(L^2)$ and $l^2(H^1)$  errors for the solutions represented in Figure \ref{fig_comp3}. }
\end{table}

Next, we consider the same tests performed in Section \ref{nse:SVMS}, but applying the feasible spectral VMS method. 

Firstly, we check the behaviour of the feasible spectral VMS method \eqref{eq:SVMS_F} for very large P\'eclet numbers.  In Figure \ref{fig_comp_1} we represent the solution of same problem as in Figure \ref{fig_test1}
obtained with this method. We show solutions in time-steps 1 to 4 in $(a)$, times-steps 5 to 7 in $(b)$ and time-steps 8 and 9 in $(c)$. As we can observe, the spectral method is the closest to the reference solution without presenting any spurious oscillations.
%
\begin{figure}[h!]
\centering
\small $(a)$ \hspace{8cm} $(b)$
\normalsize
\includegraphics[width=8.42cm]{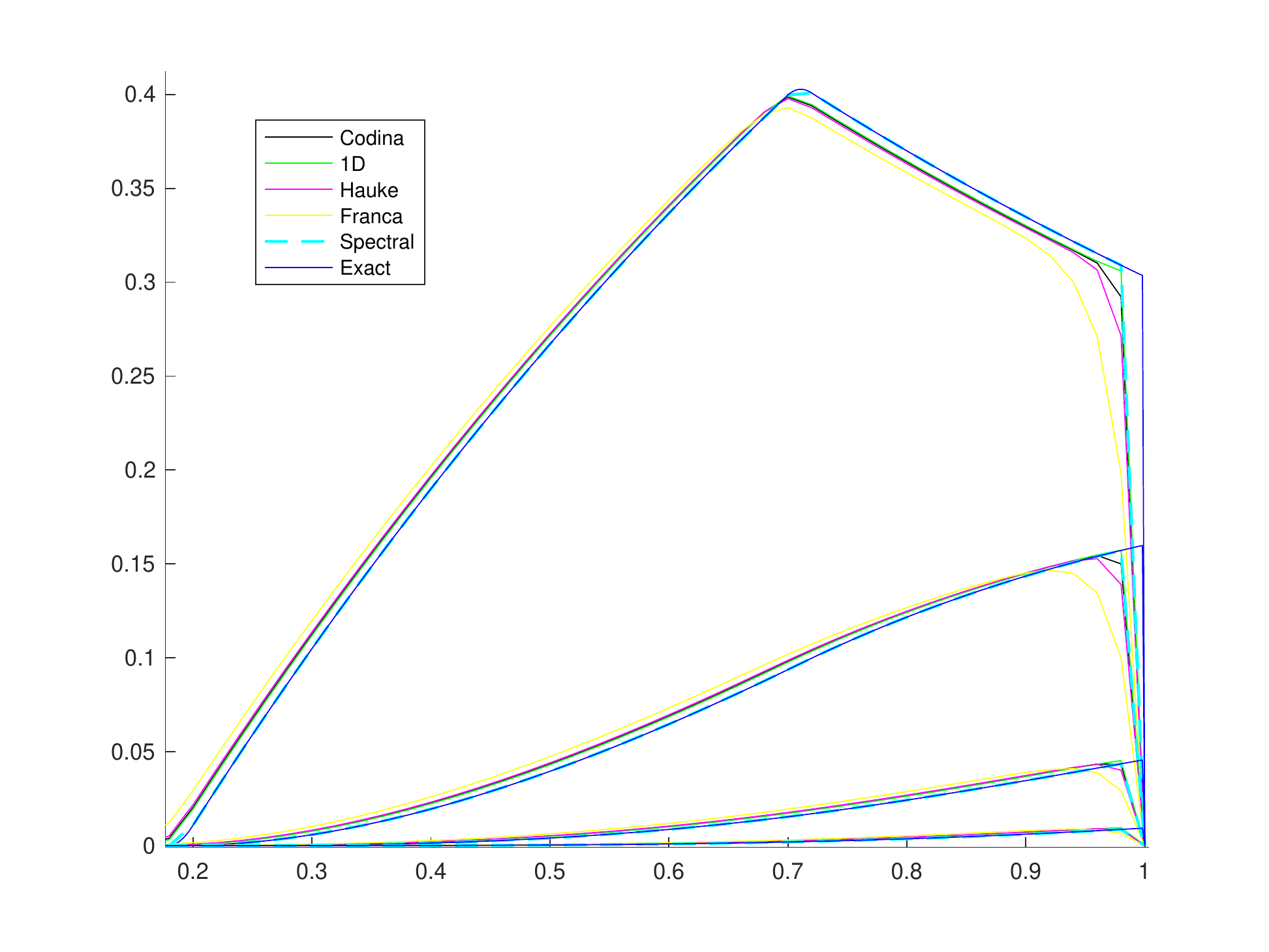}
\hfill
\includegraphics[width=8.42cm]{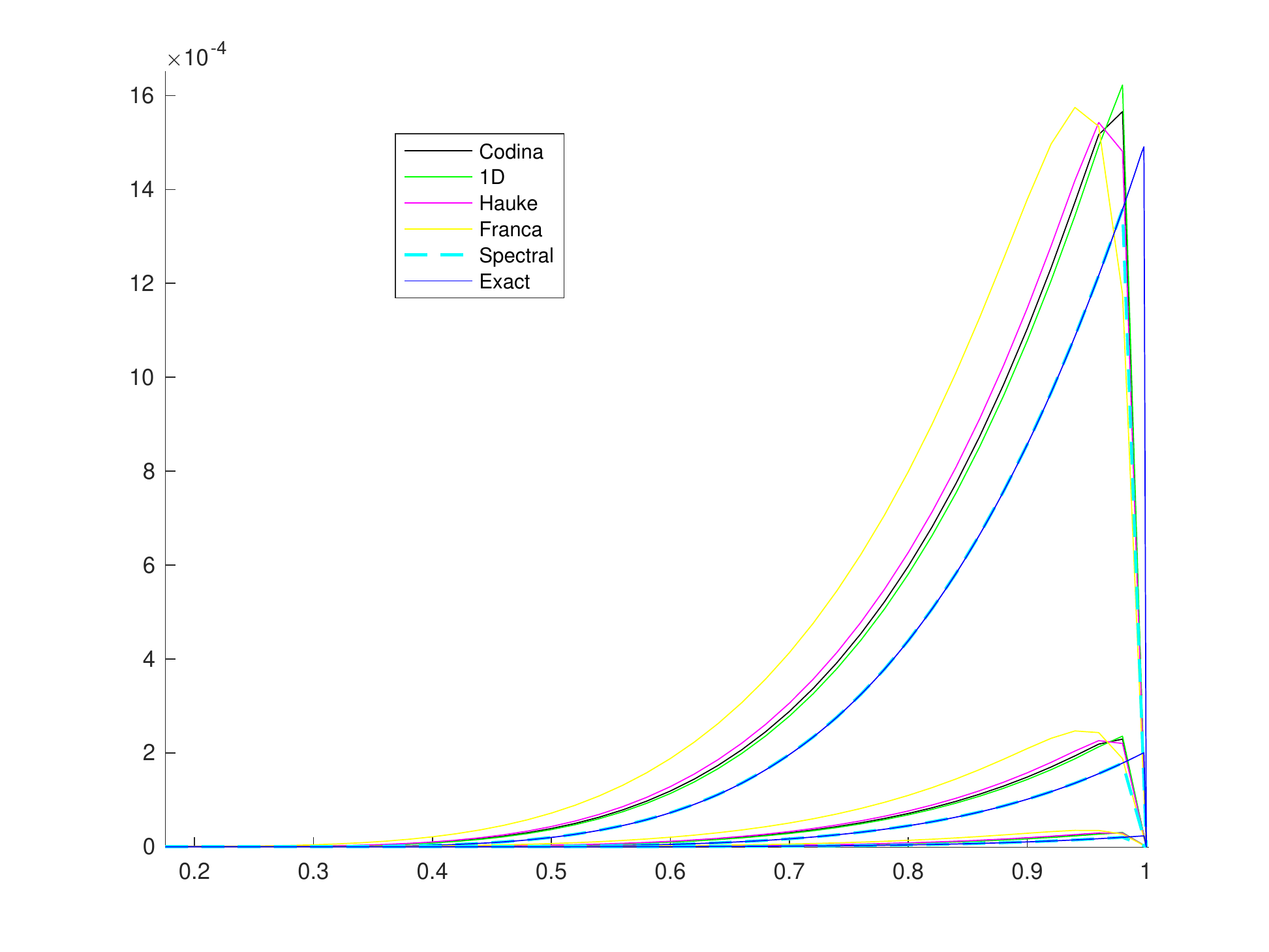}
$(c)$
\vfill
\includegraphics[width=9cm]{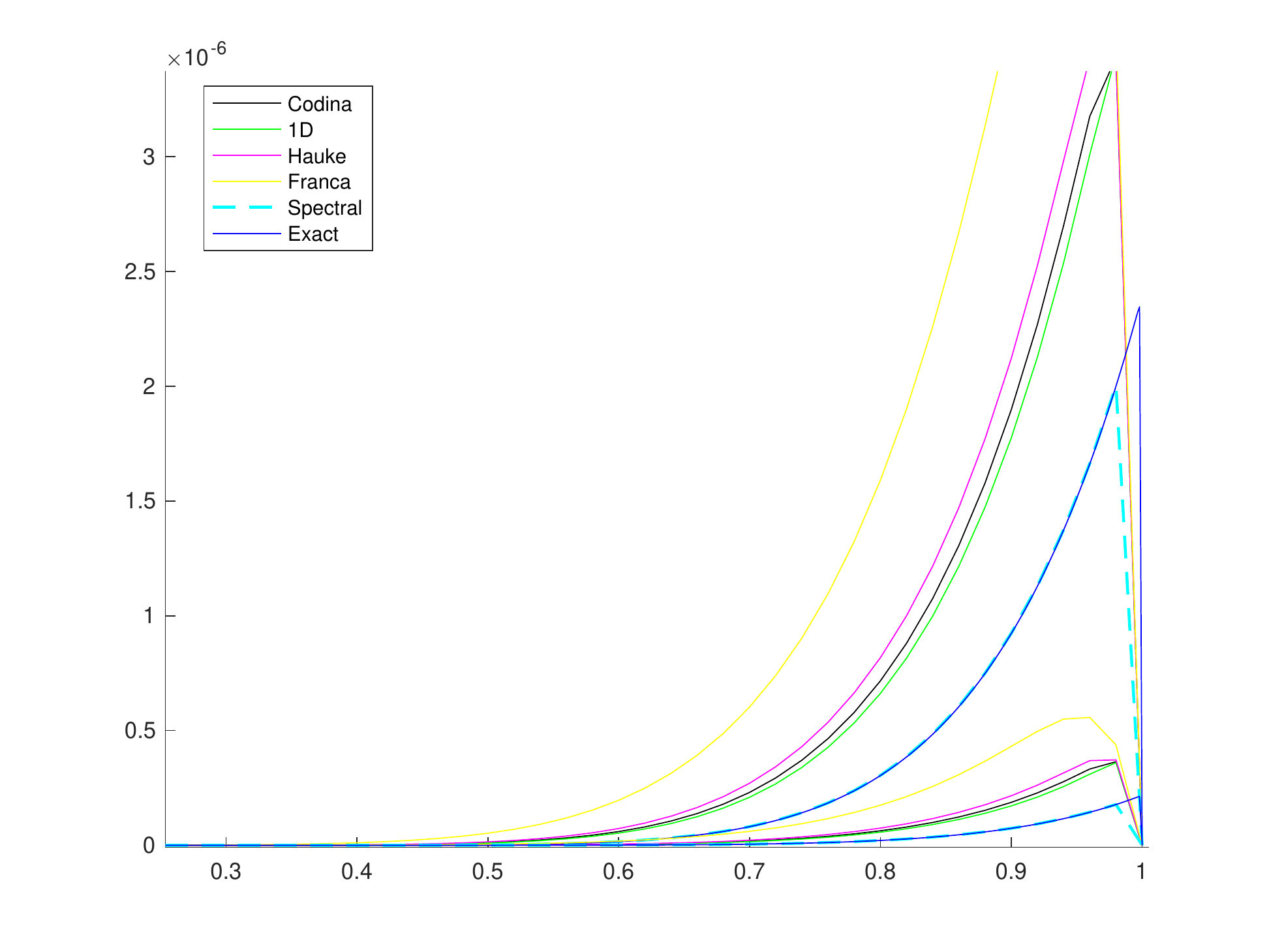}
\caption{Solution of problem \eqref{eq:cd} for $a=1000,  \mu = 1,  f=0$ and $u_0$ given by (\ref{solexacta}) with
$\Delta t=10^{-3}$ and $h=0.02$ ($P = 10$, $S = 2.5$). The feasible spectral VMS is compared with different stabilised methods. The results for time-steps numbers 1 to 4, 5 to 7 and 8 to 9 are respectively represented in figures $(a)$, $(b)$ and $(c)$.}
\label{fig_comp_1}
\end{figure}

Secondly, Figure \ref{fig_comp_2} is the analogous to Figure \ref{fig_test2}, but comparing the feasible spectral VMS with different stabilised methods. Although Hauke's solution is closer to the exact solution than the spectral method, we can see on the right figure that this approximation does not satisfy the Maximun Principle.
\begin{figure}[h!]
\centering
\includegraphics[width=8.42cm]{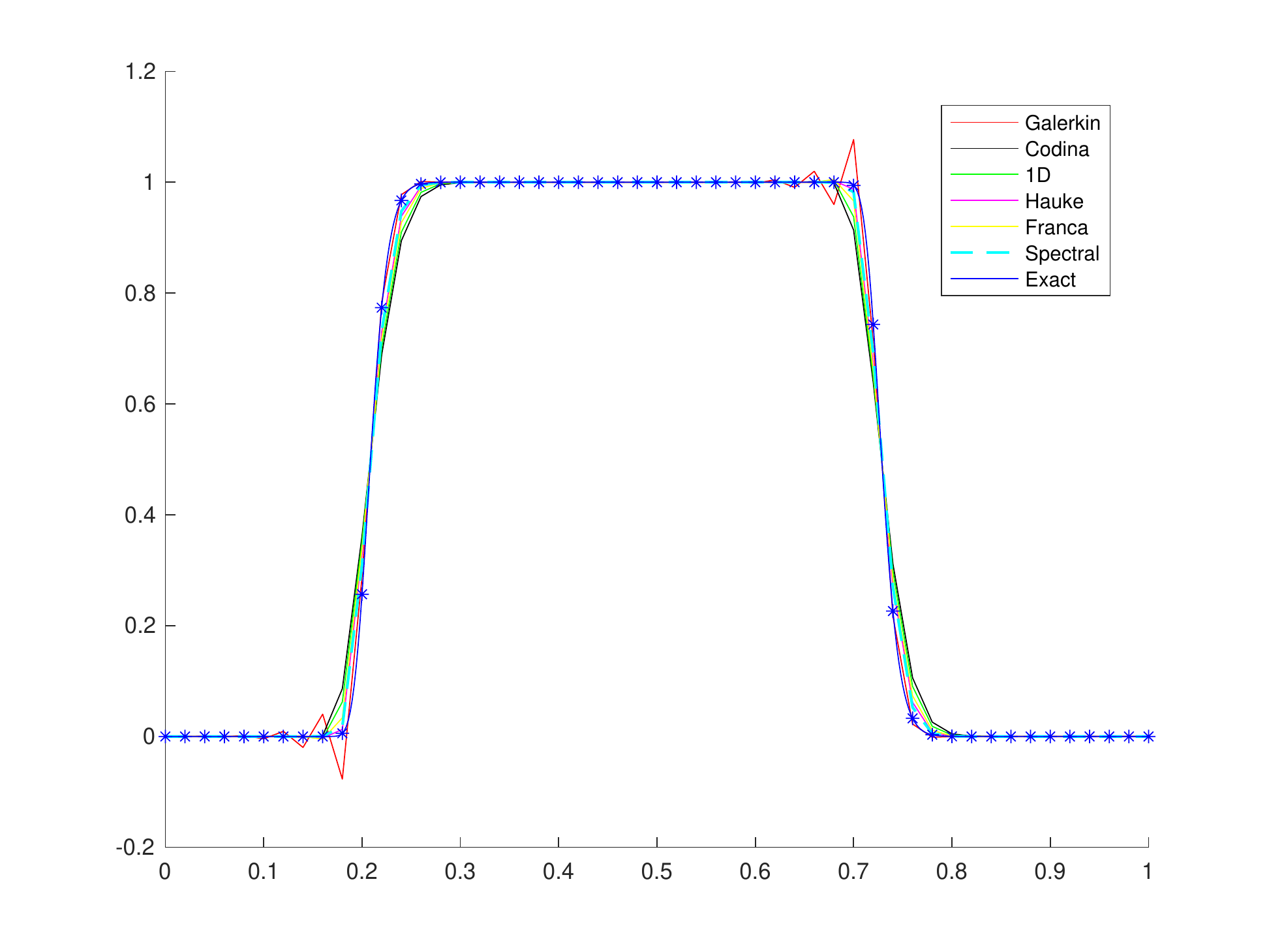}
\hfill
\includegraphics[width=8.42cm]{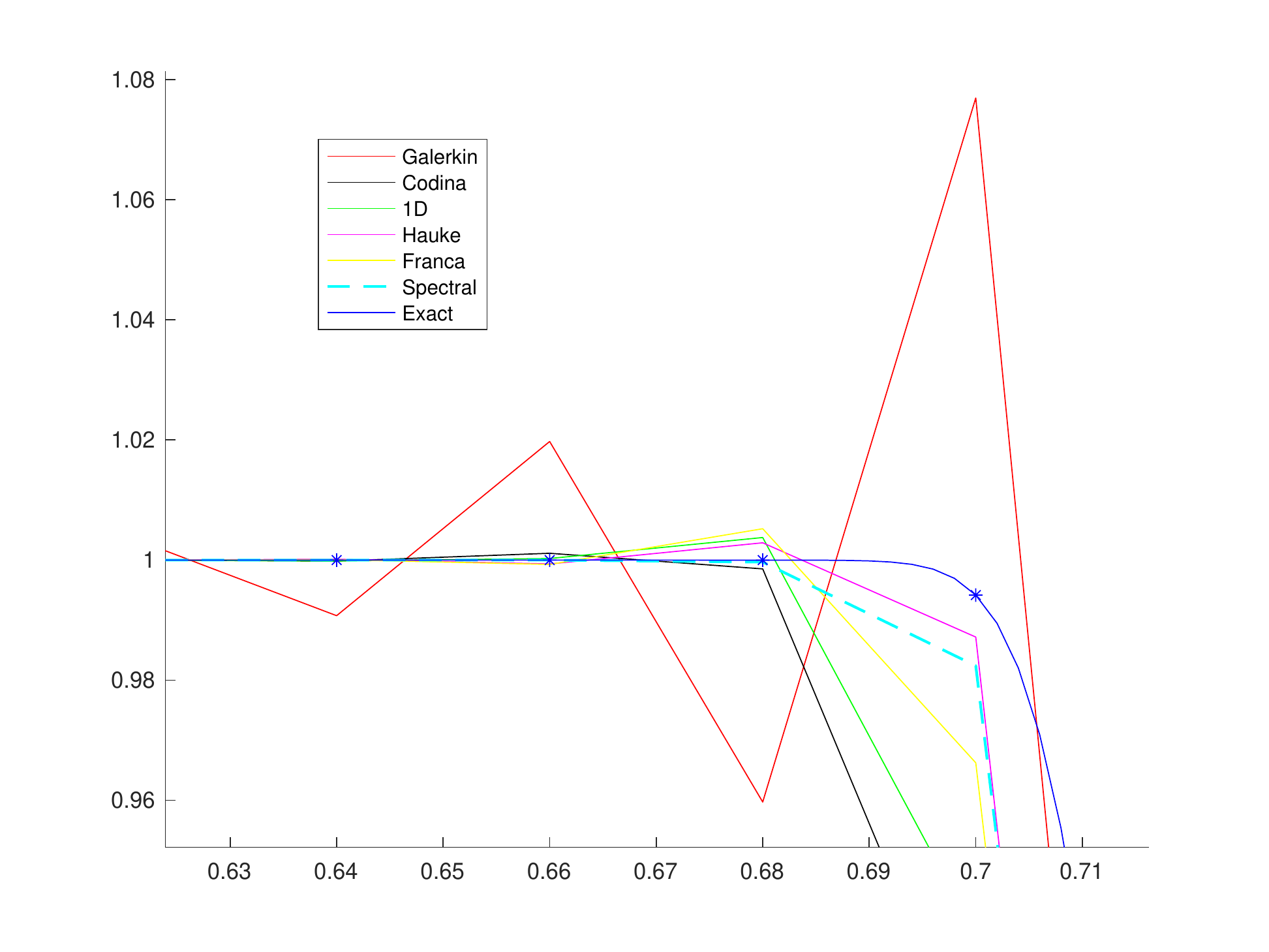}
\caption{First time-step solution of problem \eqref{eq:cd} for $a=1000,  \mu = 1,  f=0$ and $u_0$ given by (\ref{solexacta}) with $\Delta t=10^{-5}$ and $h=0.02$ ($P=10$, $S=0.025$).  The feasible spectral VMS is compared with different stabilised methods in the whole domain $\Omega=(0,1)$ in $(a)$ and in a zoom around $x=0.7$ in $(b)$. }
\label{fig_comp_2}
\end{figure}

Finally, we illustrate the fact that the feasible spectral VMS method is the only method among those studied that does not have oscillations for small time steps, when $CFL <CFL_{bound}$.  In Figure \ref{fig_comp_3}, we can see the first five time-steps solutions obtained with each method using a time-step that verifies $CFL/CFL_{bound} = 1/2$.
\begin{figure}[h!]
\centering
\includegraphics[width=16cm]{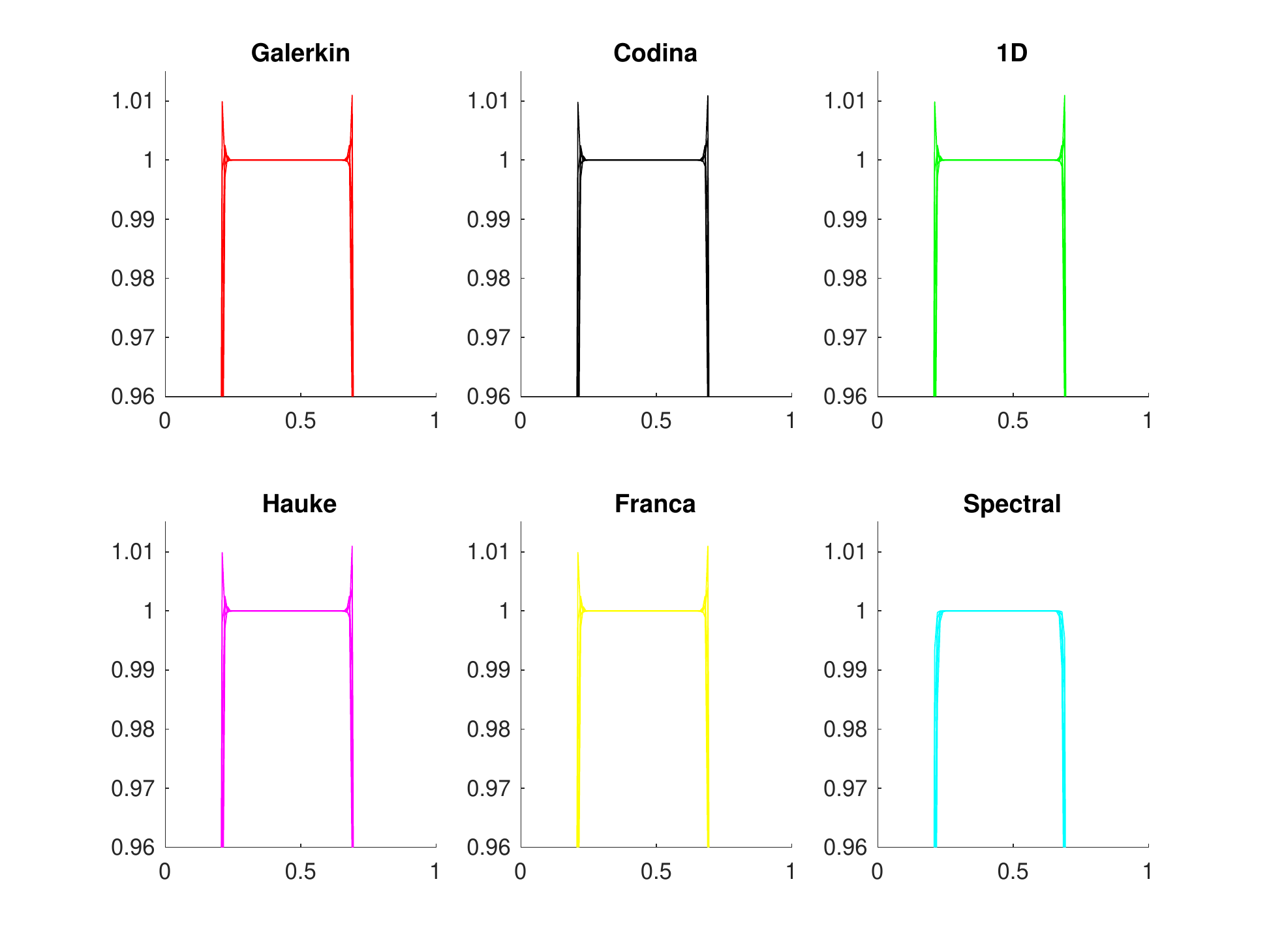}
\caption{Solution of problem \eqref{eq:cd} for $a=20,  \mu = 1,  f=0$ and $u_0$ given by (\ref{solexacta})
with $h=0.01$ and $\Delta t$ such that $CFL/CFL_{bound}= 1/2$ ($P=0.1$, $S=0.0926$).  The feasible spectral VMS is compared with different stabilised methods. }
\label{fig_comp_3}
\end{figure}

Regarding computing times, by means of the offline/online strategy the feasible spectral VMS method requires somewhat larger computing times than the remaining stabilised methods, due to the interpolation step to build the system matrices.

\section{Conclusions} \label{se:conclusions}
In this paper we have extended to parabolic problems the spectral VMS method developed in \cite{ChaconDia} for elliptic problems. We have constructed a feasible method to solve the evolutive advection-diffusion problem by means of an offline/online strategy that pre-computes the effect of the sub-grid scales on the resolved scales.

We have proved that when Lagrange finite element discretisations in space are used, the solution obtained by the fully spectral VMS method \eqref{eq:SVMS} coincides with the exact solution of the implicit Euler semi-discretisation of the advection-diffusion problem at the Lagrange interpolation nodes.

We have performed some numerical tests that have confirmed this property for very large P\'eclet numbers and very small time steps,  by the fully spectral VMS method. Also some additional tests show an improved accuracy with respect to several stabilised methods for the feasible spectral VMS method \eqref{eq:SVMS_F}, with moderate increases of computing times.

The methodology introduced here may be extended to multi-dimensional advection-diffusion equations, by parameterising the sub-grid scales in an off-line step. This research is at present in progress.

\section*{Acknowledgements} The research of T. Chac\'on and I. S\'anchez has been partially funded and that of D. Moreno fully funded  by Programa Operativo FEDER Andaluc\'{\i}a 2014-2020 grant US-1254587. The research of S. Fern\'andez has been partially funded by AEI - Feder Fund Grant RTI2018-093521-B-C31.

\newpage
\section*{Appendix: Matrix formulation of the scheme}
Problem (\ref{eq:SVMS_F}) is equivalent to a linear system with a particular structure that we describe next. If $\left\{ \varphi_m\right\}_{m=1}^{L+1}$ is a basis of the space $X_h$, the solution $u_{h}^{n+1}$ is obtained as
\begin{equation*}
u_{h}^{n+1} = \d\sum_{m=1}^{L+1} u_m^{n+1}\varphi_m,
\end{equation*}
where $\mathbf{u}^{n+1}=(u_1^{n+1},\hdots,u_{L+1}^{n+1})^t\in\R^{L+1} $ is the unknown vector.
Taking $v_h=\varphi_l$, with $l=1\ldots L$, each term in (\ref{eq:SVMS_F}) can be written in the following way:
$$
\begin{array}{l}
(u^{n+1}_h, \varphi_l) = \d\sum_{m=1}^L (\varphi_m,\varphi_l) \, u_m^{n+1} = \Big( M \, \mathbf{u}^{n+1} \Big)_l
\\[0,5cm]
b^{n+1}(u^{n+1}_h,\varphi_l) = \d\sum_{m=1}^L b^{n+1}(\varphi_m,\varphi_l) \, u_m^{n+1} = \Big( R^{n+1} \, \mathbf{u}^{n+1} \Big)_l
\\[0,5cm]
 (f^{n+1}, \varphi_l) = \Big( \mathbf{F}^{n+1} \Big)_l
\end{array}
$$
where

$(M)_{lm} =(\varphi_m,\varphi_l),$

\vspace{0.2cm}
$(R^{n+1})_{lm} =b^{n+1}(\varphi_m,\varphi_l),$
$$
\begin{array}{l}
(\utnu_h, \varphi_l) =
 \\[0,2cm]
\qquad \quad \d\sum_{m=1}^L \sum_{K\in\mathcal{T}_h} \sum_{j=1}^{\infty}
\beta_j^{n+1,K} \, (\varphi_m,p^{n+1,K} \tilde{z}_j^{n+1,K}) \, (\tilde{z}_j^{n+1,K},\varphi_l) \, u_m^{n}
\\[0,5cm]
\qquad +\d \sum_{K\in\mathcal{T}_h} \sum_{j=1}^{\infty}
 \beta_j^{n+1,K} \, (\tilde{u}_h^n,p^{n+1,K} \tilde{z}_j^{n+1,K}) \, (\tilde{z}_j^{n+1,K},\varphi_l)
\\[0,5cm]
\qquad + \, \Delta t \, \d \sum_{K\in\mathcal{T}_h} \sum_{j=1}^{\infty}
 \beta_j^{n+1,K} \, (f^{n+1},p^{n+1,K} \tilde{z}_j^{n+1,K}) \, (\tilde{z}_j^{n+1,K},\varphi_l)
\\[0,5cm]
\qquad -\d\sum_{m=1}^L \sum_{K\in\mathcal{T}_h} \sum_{j=1}^{\infty}
\beta_j^{n+1,K} \, (\varphi_m,p^{n+1,K} \tilde{z}_j^{n+1,K}) \, (\tilde{z}_j^{n+1,K},\varphi_l) \, u_m^{n+1}
\\[0,5cm]
\qquad - \, \Delta t \, \d\sum_{m=1}^L \sum_{K\in\mathcal{T}_h} \sum_{j=1}^{\infty}
 \beta_j^{n+1,K} \, b^{n+1} (\varphi_m,p^{n+1,K} \tilde{z}_j^{n+1,K}) \, (\tilde{z}_j^{n+1,K},\varphi_l) \, u_m^{n+1}
 \\[0,6cm]
\qquad =\Big( A^{n+1}_1 \, \mathbf{u}^{n} + \mathbf{G}^{n+1}_1 + \Delta t \, \mathbf{F}^{n+1}_1
-A^{n+1}_1 \, \mathbf{u}^{n+1} - \Delta t \, A^{n+1}_2 \, \mathbf{u}^{n+1} \Big)_l,
\end{array}
$$
with
\begin{eqnarray}
\label{a1}
&& (A^{n+1}_1)_{lm} = \d\sum_{K\in\mathcal{T}_h} \sum_{j=1}^{\infty} \beta_j^{n+1,K} \,
(\varphi_m, p^{n+1,K} \tilde{z}_j^{n+1,K})  \,(\tilde{z}_j^{n+1,K}, \varphi_l ),
\\[0,2cm]
\label{a2}
&& (A^{n+1}_2)_{lm}  = \d\sum_{K\in\mathcal{T}_h} \sum_{j=1}^{\infty} \beta_j^{n+1,K} \,
b^{n+1}(\varphi_m, p^{n+1,K} \tilde{z}_j^{n+1,K}) \, (\tilde{z}_j^{n+1,K}, \varphi_l ), 
\\[0,2cm]
&& (\mathbf{F}^{n+1}_1)_l =
 \d \sum_{K\in\mathcal{T}_h} \sum_{j=1}^{\infty}
 \beta_j^{n+1,K} \, \langle f^{n+1},p^{n+1,K} \tilde{z}_j^{n+1,K} \rangle \, (\tilde{z}_j^{n+1,K},\varphi_l), \nonumber
\\[0,2cm]
&& (\mathbf{G}^{n+1}_1)_l =  \d \sum_{K\in\mathcal{T}_h} \sum_{j=1}^{\infty}
 \beta_j^{n+1,K} \, (\tilde{u}_h^n,p^{n+1,K} \tilde{z}_j^{n+1,K}) \, b^{n+1}(\tilde{z}_j^{n+1,K},\varphi_l),\nonumber
\end{eqnarray}
and
$$
\begin{array}{l}
b^{n+1}(\utnu_h, v_h) =
\\[0,2cm]
\qquad \quad \d\sum_{m=1}^L \sum_{K\in\mathcal{T}_h} \sum_{j=1}^{\infty}
\beta_j^{n+1,K} \, (\varphi_m,p^{n+1,K} \tilde{z}_j^{n+1,K}) \, b^{n+1}(\tilde{z}_j^{n+1,K},\varphi_l) \, u_m^{n}
\\[0,5cm]
\qquad + \d \sum_{K\in\mathcal{T}_h} \sum_{j=1}^{\infty}
\beta_j^{n+1,K} \, (\tilde{u}_h^n,p^{n+1,K} \tilde{z}_j^{n+1,K}) \, b^{n+1}(\tilde{z}_j^{n+1,K},\varphi_l)
\\[0,5cm]
\qquad + \, \Delta t \, \d \sum_{K\in\mathcal{T}_h} \sum_{j=1}^{\infty}
\beta_j^{n+1,K} \, (f^{n+1},p^{n+1,K} \tilde{z}_j^{n+1,K}) \, b^{n+1}(\tilde{z}_j^{n+1,K},\varphi_l)
\\[0,5cm]
\qquad -\d\sum_{m=1}^L \sum_{K\in\mathcal{T}_h} \sum_{j=1}^{\infty}
 \beta_j^{n+1,K} \, (\varphi_m,p^{n+1,K} \tilde{z}_j^{n+1,K}) \, b^{n+1}(\tilde{z}_j^{n+1,K},\varphi_l) \, u_m^{n+1}
 \\[0,5cm]
\qquad - \, \Delta t \, \d\sum_{m=1}^L \sum_{K\in\mathcal{T}_h} \sum_{j=1}^{\infty}
 \beta_j^{n+1,K} \, b^{n+1} (\varphi_m,p^{n+1,K} \tilde{z}_j^{n+1,K}) \, b^{n+1}(\tilde{z}_j^{n+1,K},\varphi_l) \, u_m^{n+1}  \\[0,6cm]
\qquad =\Big(A^{n+1}_3 \, \mathbf{u}^{n} + \mathbf{G}^{n+1}_2+ \Delta t \, \mathbf{F}^{n+1}_2 -
 A^{n+1}_3 \, \mathbf{u}^{n+1} - \Delta t \,A^{n+1}_4 \, \mathbf{u}^{n+1} \Big)_l,
 \end{array}
$$
where
\begin{eqnarray}
\label{a3}
&& (A^{n+1}_3)_{lm} = \d\sum_{K\in\mathcal{T}_h} \sum_{j=1}^{\infty} \beta_j^{n+1,K} \,
(\varphi_m, p^{n+1,K} \tilde{z}_j^{n+1,K}) \, b^{n+1}(\tilde{z}_j^{n+1,K}, \varphi_l ),
\\[0,2cm]
\label{a4}
&& (A^{n+1}_4)_{lm} =\d\sum_{K\in\mathcal{T}_h} \sum_{j=1}^{\infty} \beta_j^{n+1,K} \,
b^{n+1}(\varphi_m, p^{n+1,K} \tilde{z}_j^{n+1,K} \, b^{n+1}(\tilde{z}_j^{n+1,K}, \varphi_l ),
\\[0,2cm]
&& (\mathbf{F}^{n+1}_2)_l  =
 \d \sum_{K\in\mathcal{T}_h} \sum_{j=1}^{\infty}
 \beta_j^{n+1,K} \, \langle f^{n+1},p^{n+1,K} \tilde{z}_j^{n+1,K} \rangle \, b^{n+1}(\tilde{z}_j^{n+1,K},\varphi_l),\nonumber
\\[0,2cm]
&&(\mathbf{G}^{n+1}_2)_l  =  \d \sum_{K\in\mathcal{T}_h} \sum_{j=1}^{\infty}
 \beta_j^{n+1,K} \, (\tilde{u}_h^n,p^{n+1,K} \tilde{z}_j^{n+1,K}) \, b^{n+1}(\tilde{z}_j^{n+1,K},\varphi_l).\nonumber
\end{eqnarray}

Taking in account the definition of $\tilde{u}_h^n$ in \eqref{eq:uhtn_F}, the second terms $\mathbf{G}^{n+1}_1$ and 
$\mathbf{G}^{n+1}_2$ can be expressed in the following way: 
$$
\begin{array}{l}
\Big(\mathbf{G}^{n+1}_1 \Big)_l = \d \sum_{K\in\mathcal{T}_h} \sum_{j=1}^{\infty} \beta_j^{n+1,K} \, 
\beta_j^{n,K} \, \langle \hat{R}^n_h({u}^n_h),  p^{n,K} \, \tilde{z}_j^{n,K} \rangle (\tilde{z}_j^{n+1,K},\varphi_l \Big)
\\ [0,6cm]
\qquad = \Big(  B_1^{n+1}  \, \mathbf{u}^{n-1}  +  \Delta t \, \mathbf{F}^{n+1}_3  
- B_1^{n+1}  \, \mathbf{u}^{n} - \Delta t \,  B_2^{n+1}  \, \mathbf{u}^{n}
\Big)_l,
\\ [0,5cm]
\Big(\mathbf{G}^{n+1}_2 \Big)_l = \d \sum_{K\in\mathcal{T}_h} \sum_{j=1}^{\infty} \beta_j^{n+1,K} \, 
\beta_j^{n,K} \, \langle \hat{R}^n_h({u}^n_h),  p^{n,K} \, \tilde{z}_j^{n,K} \rangle \, b^{n+1}(\tilde{z}_j^{n+1,K},\varphi_l \Big)
\\ [0,6cm]
\qquad = \Big(  B_3^{n+1}  \, \mathbf{u}^{n-1}  +  \Delta t \, \mathbf{F}^{n+1}_4  
- B_3^{n+1}  \, \mathbf{u}^{n} - \Delta t \,  B_4^{n+1}  \, \mathbf{u}^{n}
\Big)_l,
\end{array}
$$
where
\begin{eqnarray}
\label{b1}
&&(B^{n+1}_1)_{lm} = \d\sum_{K\in\mathcal{T}_h} \sum_{j=1}^{\infty} \beta_j^{n+1,K} \, \beta_j^{n,K} \,
(\varphi_m, p^{n,K} \tilde{z}_j^{n,K})  \,(\tilde{z}_j^{n+1,K}, \varphi_l ),
\\[0,2cm]
\label{b2}
&&(B^{n+1}_2)_{lm} = \d\sum_{K\in\mathcal{T}_h} \sum_{j=1}^{\infty} \beta_j^{n+1,K} \, \beta_j^{n,K} \,
b^n(\varphi_m, p^{n,K} \tilde{z}_j^{n,K})  \, (\tilde{z}_j^{n+1,K}, \varphi_l ),
\\[0,2cm]
\label{b3}
&& (B^{n+1}_3)_{lm} = \d\sum_{K\in\mathcal{T}_h} \sum_{j=1}^{\infty} \beta_j^{n+1,K} \, \beta_j^{n,K} \,
(\varphi_m, p^{n,K} \tilde{z}_j^{n,K})  \, b^{n+1} (\tilde{z}_j^{n+1,K}, \varphi_l ),
\\[0,2cm]
\label{b4}
&& (B^{n+1}_4)_{lm} = \d\sum_{K\in\mathcal{T}_h} \sum_{j=1}^{\infty} \beta_j^{n+1,K} \, \beta_j^{n,K} \,
b^n(\varphi_m, p^{n,K} \tilde{z}_j^{n,K})  \, b^{n+1} (\tilde{z}_j^{n+1,K}, \varphi_l ),
\\[0,2cm]
&&(\mathbf{F}^{n+1}_3)_l  =
 \d \sum_{K\in\mathcal{T}_h} \sum_{j=1}^{\infty}
 \beta_j^{n+1,K} \, \beta_j^{n,K} \, \langle f^{n},p^{n,K}  \tilde{z}_j^{n,K} \rangle \, (\tilde{z}_j^{n+1,K},\varphi_l),\nonumber
\\[0,2cm]
&& (\mathbf{F}^{n+1}_4)_l  =
 \d \sum_{K\in\mathcal{T}_h} \sum_{j=1}^{\infty}
 \beta_j^{n+1,K} \, \beta_j^{n,K} \, \langle f^{n},p^{n,K} \tilde{z}_j^{n,K} \rangle \, 
 b^{n+1}(\tilde{z}_j^{n+1,K},\varphi_l).\nonumber
\end{eqnarray}

\vspace{0.5cm}

Here we are neglecting the interaction between different eigenfunctions in two consecutive time steps. Obviously, this occurs when the operator is time independent.

\vspace{0.5cm}
\noindent Thus, problem (\ref{eq:SVMS_F}) is equivalent to the lineal system
$$
\mathbf{A}^{n+1} \, \mathbf{u}^{n+1} = \mathbf{b}^{n+1},
$$
where $\mathbf{A}^{n+1}\in \R^{(L+1)\times (L+1)}$ and $\mathbf{b}^{n+1}\in \R^{(L+1)}$ are given by
\begin{equation}
\label{eq:csl}
\barr{l}
\mathbf{A}^{n+1}= M + \Delta t \, R^{n+1} - {\cal A}^{n+1},
\\[0,5cm]
\mathbf{b}^{n+1} =
\Big(  M - (A^{n+1}_1  + \Delta t \, A^{n+1}_3) - (A^{n}_1 + \Delta t \, A^{n}_2) - {\cal B}^{n+1} \Big) \mathbf{u}^n\nonumber
\\[0,4cm]
\qquad + \Big( A^{n}_1 - B_1^{n+1} - \Delta t \, B_3^{n+1} \Big) \, \mathbf{u}^{n-1}
\\[0,5cm]
\qquad + \Delta t \, \mathbf{F}^{n+1} - \Delta t \, \mathbf{F}^{n+1}_1 - \Delta t^2 \, \mathbf{F}^{n+1}_2 +
   \Delta t \, \mathbf{F}^{n}_1 - \Delta t \, \mathbf{F}^{n+1}_3 - \Delta t^2 \, \mathbf{F}^{n+1}_4,
\earr
\end{equation}
with 
$${\cal A}^{n+1} = A^{n+1}_1 + \Delta t \, A^{n+1}_2 + \Delta t \, A^{n+1}_3  + \Delta t^2 \, A^{n+1}_4, \qquad 
{\cal B}^{n+1} = B_1^{n+1} - \Delta t \, B_2^{n+1} - \Delta t \, B_3^{n+1} - \Delta t^2 \, B_4^{n+1}.$$

Here, $M$ and $R^{n+1}$ are, respectively, the mass and stiffness matrices from the Galerkin formulation and 
$A_i^{n+1}$ and $B_i^{n+1}$ are the matrices that represent the effect of the small scales component of the solution on the large scales component.

\end{document}